\documentclass[a4paper,12pt,oneside,reqno]{amsart}
\usepackage{hyperref}
\usepackage[headinclude,DIV13]{typearea}
\areaset{15.1cm}{25.0cm}
\usepackage{amsfonts,amssymb,amsmath,amsthm,bbm}
\usepackage[latin1] {inputenc}
\usepackage{graphicx, psfrag}
\usepackage{subfigure}

\newtheorem{theorem}{Theorem}[section]
\newtheorem{lemma}[theorem]{Lemma}
\newtheorem{proposition}[theorem]{Proposition}
\newtheorem{corollary}[theorem]{Corollary}

\theoremstyle{definition}

\theoremstyle{remark}
\newtheorem*{remark}{Remark}

\def\paragraph#1{\noindent \textbf{#1}}

\numberwithin{equation}{section}

\def\d{\mathrm{d}}

\def\<{\langle}
\def\>{\rangle}

\def\e{\epsilon}

\def\g{\gamma}

\def\s{\sigma}

\def\o{\omega}
\def\D{\Delta}
\def\L{\Lambda}

\def\del{\partial}
\def\R{{\Bbb R}}  
\def\N{{\Bbb N}}  
\def\P{{\Bbb P}}

\def\E{{\Bbb E}}

\let\cal=\mathcal

\def\CC{{\cal C}}
\def\DD{{\cal D}}
\def\EE{{\cal E}}
\def\FF{{\cal F}}
\def\GG{{\cal G}}

\def\NN{{\cal N}}

\def\TT{{\cal T}}

 \def \L {{\Lambda}}

 \def \s {{\sigma}}

 \def \D {{\Delta}}

 \def \g {{\gamma}}
 
 \def \d {{\delta}}
 
 \def \o {{\omega}}

 \def \del {{\partial}}

 \def \ba {\begin{array}}
 \def \ea {\end{array}}

 \newcommand{\be}{\begin{equation}}
 \newcommand{\ee}{\end{equation}}

\newcommand{\bea}{\begin{eqnarray}}
 \newcommand{\eea}{\end{eqnarray}}
\def\TH(#1){\label{#1}}\def\thv(#1){\ref{#1}}
\def\Eq(#1){\label{#1}}\def\eqv(#1){(\ref{#1})}

\def\sfrac#1#2{{\textstyle{#1\over #2}}}

 \def \1{\mathbbm{1}}
\def\wt {\widetilde}

\begin{document}

 \title[The extremal process of two-speed BBM]
{The extremal process of \\ two-speed branching Brownian motion}
\author[A. Bovier]{Anton Bovier}
 \address{A. Bovier\\Institut f\"ur Angewandte Mathematik\\
Rheinische Friedrich-Wilhelms-Universität\\ Endenicher Allee 60\\ 53115 Bonn, Germany }
\email{bovier@uni-bonn.de}
\author[L. Hartung]{Lisa Hartung}
 \address{L. Hartung\\Institut f\"ur Angewandte Mathematik\\
Rheinische Friedrich-Wilhelms-Universität\\ Endenicher Allee 60\\ 53115 Bonn, Germany}
\email{lhartung@uni-bonn.de}

\subjclass[2000]{60J80, 60G70, 82B44} \keywords{branching Brownian motion, 
extremal processes, extreme values, F-KPP equation, cluster point processes} 

\date{\today}

 \begin{abstract} We construct and describe the extremal process for
 variable speed branching Brownian motion, studied recently by Fang and Zeitouni \cite{FZ_BM}, for the case of piecewise constant speeds; in fact for simplicity we concentrate on the case when the speed is $\s_1$ for $s\leq bt$ and $\s_2$ when $bt\leq s\leq t$. In the case $\s_1>\s_2$, the process is the concatenation of two BBM extremal processes, as expected. In the case $\s_1<\s_2$, a new family  of cluster point processes arises, that are similar, but distinctively different from the BBM process. 
 Our proofs follow the strategy of Arguin, Bovier, and Kistler in \cite{ABK_E}.

 \end{abstract}

\thanks{A.B. is partially supported through the German Research Foundation in 
the Collaborative Research Center 1060 "The Mathematics of Emergent Effects", 
the Priority Programme  1590 ``Probabilistic Structures in Evolution'',
the Hausdorff Center for Mathematics (HCM), and  the Cluster of Excellence ``ImmunoSensation'' at Bonn University.
L.H. is supported by the German Research Foundation in the Bonn International Graduate School in Mathematics (BIGS). 
}

 \maketitle

\section{Introduction}

A standard branching Brownian motion (BBM) is a continuous-time Markov branching process that is constructed as follows: start with a single particle which performs a standard Brownian motion $x(t)$ with $x(0)=0$ and continues for a standard exponentially distributed  holding time $T$,
independent of $x$. At time $T$, the particle splits independently of $x$ and $T$ into $k$ offspring with probability $p_k$, where $\sum_{i=1}^\infty p_k=1$,
$\sum_{k=1}^\infty k p_k=2$ and $K=\sum_{k=1}^\infty k(k-1)p_k<\infty$. These particles continue along independent Brownian paths starting from $x(T)$ and are subject to the same splitting rule.
And so on.

Branching Brownian motion has received a lot of attention over the last decades, with a strong focus on the properties of extremal particles. We mention the seminal contributions of McKean \cite{McKean}, Bramson, Lalley and Sellke, and Chauvin and Rouault \cite{B_M,B_C,LS, chauvin90} on the connection to the Fisher-Kolmogorov-Petrovsky-Piscounov (F-KPP) equation and on the distribution of the rescaled maximum. In recent years, their has been a revival of interest in BBM with numerous contributions, including the construction of the full extremal process \cite{ABK_E, abbs}.
For a review of these developments see, e.g., the recent survey by Gu\'er\'e \cite{gouere}.

BBM can be seen as a Gaussian process with covariances depending on an ultrametric 
distance, in this case the ultrametric associated to the genealogical structure of an 
underlying Galton-Watson process. In that respect it is closely related to another class of Gaussian processes, the Generalised Random Energy Models (GREM) introduced by Derrida and Gardner \cite{GD86b}.  While in BBM the covariance of the process is a linear function of the ultrametric distance, in the GREM one considers more general functions. One of the reasons that makes BBM interesting in this context is the fact that the linear function appears as a  borderline where the correlation starts to modify the behaviour of extremes \cite{BK1, BK2}. 

In the context of BBM, different covariances can be  achieved by varying the speed (i.e. diffusivity) of the Brownian motions as a function of time (see also \cite{BK2}). This model was introduced by Derrida and Spohn \cite{DerriSpohn88} and has  recently been investigated by 
Fang and Zeitouni \cite{FZ_BM, FZ_RW}, see also \cite{MZ,Mallein}. 
The entire family of models obtained as time changes of BBM is a 
splendid test ground to further develop the theory of extremes of 
correlated random variables. Understanding fully the possible extremal processes that arise in this class should also provide us with candidate processes for even wider classes of random structures. 

\subsection{Results}

In \cite{FZ_BM}, Fang and Zeitouni showed that in the case when the covariance is a piecewise linear function, the maximum of BBM is tight and behaves as expected from the analogous GREM. 
In this paper we refine and extend their analysis: we obtain the precise law of the maximum, and we give the full characterisation of the extremal process. 

For simplicity  we consider the following variable speed BBM. Fix a time $t$. Then we consider the BBM model where at time $s$, all particles move independently as Brownian motions with variance
\be\Eq(speed.1)
\s^2(s)=\left\{\begin{array}{ll}
    \s_1^2 &  \quad 0\leq s<bt \\
  \s_2^2         &  \quad  t \leq s \leq t
  \end{array}
\right. ,\quad  0<b\leq1.
\ee
We normalise the total variance by assuming
\be\Eq(norm.1)
\s_1^2b+\s_2^2(1-b)=1.
\ee
Note that in the case $b=1$, $\s_2=\infty$ is allowed.

We denote by $n(s)$ the number of particles at time $s$
 and by  $\{x_i(s);1\leq i \leq n(s)\}$ the positions of the particles at time $s$. 
\begin{remark} Strictly speaking, we are not talking about a single stochastic process, but about a family $\{x_k(s),k\leq n(s)\}_{s\leq t}^{t\in \R_+}$ of processes with finite time horizon, indexed by that horizon, $t$. 
\end{remark}

In this case, Fang and Zeitouni \cite{FZ_RW}  showed that
\be
\Eq(fz.1)
\max_{k\leq n(t)} x_k(t)=
\begin{cases} \sqrt 2 t -\frac{1}{2\sqrt 2} \log t +O(1),&\,\hbox{if} \,\,\s_1<\s_2,\\
 \sqrt{2} t (b\s_1+(1-b)\s_2)  - \frac{3}{2\sqrt 2} (\s_1+\s_2)   \log t+ O(1),&\,\text{if} \,\,\s_1>\s_2.
\end{cases}
\ee
The second case has a simple interpretation: the maximum is achieved by adding to the maxima of BBM at time $bt$ the maxima of their offspring at time $(1-b)t$ later. 
The first case looks simpler even, but is far more interesting. The order of the maximum 
is that of the REM, a fact to be expected by the corresponding results in the GREM (see
\cite {GD86b, BK1}). But what is the law of the rescaled maximum and what is the 
corresponding extremal process? The purpose of this paper is primarily to answer this question.

For standard BBM, $\bar x(t)$,  (i.e. $\s_1=\s_2$), 
Bramson \cite{B_M} and Lalley and Sellke \cite{LS} show that
\be\Eq(old.1)
\lim_{t\uparrow \infty} \P\left(\max_{k\leq n(t)}\bar x_k(t)-m(t) \leq y\right)
=\o(x) =\E e^{-CZe^{-\sqrt 2 y}},
\ee
where $m(t)\equiv \sqrt 2 t-\frac{3}{2\sqrt 2} \log t$, $Z$ is a random variable, the limit of the so called \emph{derivative martingale}, and $C$ is a constant.

In \cite{ABK_E} (see also 
\cite{abbs}  for a different proof) it was shown that the extremal process, 
\be 
\Eq(old.2)
\lim_{t\uparrow\infty} \wt\EE_t\equiv \lim_{t\uparrow\infty}\sum_{k=1}^{n(t)} \d_{\bar x_k(t)-m(t)}=\wt\EE,
\ee
exists in law, and $\wt\EE$ is of the form 
\be\Eq(old.3)
\wt\EE =\sum_{k,j}\d_{\eta_k+\D^{(k)}_j},
\ee
where $\eta_k$ is the $k$-th atom of a mixture of Poisson point process with intensity 
measure $CZ e^{-\sqrt 2 y}dy$, with $C$ and $Z$ as before, 
and
$\D^{(k)}_i$ are the atoms of independent and identically distributed  point processes $\D^{(k)}$, which are the limits in law
of 
\be
\sum_{j\leq n(t) }\d_{\tilde x_i(t)-\max_{j\leq n(t)}\tilde x_j(t)},
\Eq(old.4)
\ee 
where $\tilde x(t)$ is BBM conditioned on $\max_{j\leq n(t)} \tilde x_j(t)\geq \sqrt 2 t$.

The main result of the present paper is similar but different.

\begin{theorem} Let $x_k(t)$ be branching Brownian motion with variable speed  $\s^2(s)$ as given 
in \eqv(speed.1).  Assume that $\s_1<\s_2$. Then 
\begin{itemize} 
\item[(i)]
\be
\Eq(main.1)
\lim_{t\uparrow \infty} \P\left(\max_{k\leq n(t)} x_k(t)-\tilde m(t) \leq y\right)
=\E e^{-C' Y e^{-\sqrt 2 y}},
\ee 
where $\tilde m(t)=\sqrt 2 t-\frac{1}{2\sqrt 2} \log t$, $C'$ is a constant and $Y$ is a random variable that is the limit of a martingale
(but different from $Z$!). 
\item [(ii)] The point process 
\be\Eq(main.2)
\EE_t\equiv \sum_{k\leq n(t)} \d_{x_k(t)-\tilde m(t)}\rightarrow \EE,
\ee
as $t\uparrow \infty$, in law, where 
\be\Eq(main.3)
\EE =\sum_{k,j}\d_{\eta_k+\s_2\L^{(k)}_j},
\ee
where $\eta_k$ is the $k$-th atom of a mixture of Poisson point process with intensity 
measure $C'Y e^{-\sqrt 2 y}dy$, with $C'$ and $Y$ as in (i),  
and
$\L^{(k)}_i$ are the atoms of independent and identically distributed  point processes $\L^{(k)}$, which are the limits in law
of 
\be
\sum_{j\leq n(t) }\d_{\tilde x_i(t)-\max_{j\leq n(t)}\tilde x_j(t)},
\Eq(main.4)
\ee 
where $\tilde x(t)$ is BBM of speed $1$ conditioned on $\max_{j\leq n(t)} \tilde x_j(t)\geq \sqrt 2\s_2 t$.
\end{itemize}
\end{theorem}
To complete the picture, we give the result for the limiting extremal process in the case
$\s_1>\s_2$. This result is much simpler and totally unsurprising. 

\begin{theorem} \TH(speed.2) 
 Let $x_k(t)$ be as in Theorem \thv(speed.1), but $\s_2<\s_1$. 
 Let $\wt\EE\equiv \wt\EE^0$ and $\wt\EE^{(i)}, i\in \N$ be independent copies of the 
 extremal process \eqv(old.3) of standard branching Brownian motion. 
 Let 
 \be
 m(t)\equiv\sqrt{2} t (b\s_1+(1-b)\s_2)-\frac {3}{2\sqrt {2}} (\s_1+\s_2)\log t
 -\frac {3}{2\sqrt {2}}(\s_1 \log b+\s_2 \log(1-b)),
 \ee
  and set 
 \be\Eq(speed.3)
  \EE_t \equiv \sum_{k\leq n(t)} \d_{x_k(t)-m(t)}.
 \ee
 Then 
 \be\Eq(speed.4) 
 \lim_{t\uparrow\infty} \EE_t= \EE,
 \ee
 exists in law, and 
 \be\Eq(speed.5)
 \EE =\sum_{i,j} \d_{\s_1 e_i+\s_2 e^{(i)}_j},
 \ee
 where $e_i, e^{(i)}_j$ are the atoms of the point processes $\wt\EE$ and 
 $\wt \EE^{(i)}$, respectively.
 \end{theorem}

 \begin{remark} In the case $\s_1<1$, we see that the limiting process depends only on the values of $\s_1$ (through the martingale $Y$) and on $\s_2$ (through the processes of clusters $\s_2
 \L^{(k)}$). As $\s_2$ grows, the clusters become spread out, and in the limit 
 $\s_2=\infty$, the cluster processes degenerate to the Dirac mass at zero. Hence, in that case
 the extremal process is just a mixture of Poisson point processes. 
  When $\s_1=0$, and $b>0$, the martingale limit is just an exponential random variable, the limit of
 the martingale $n(t)e^{-t}$. The case $b=1$, $\s_1=0$ 
 corresponds to the random REM, where there is just a random number of iid random variables of variance one present.  
 \end{remark}
 
\begin{remark}  We have decided to write this paper only for the case of two speeds. 
 It is fairly straightforward to extend our results to the general case of piecewise constant speed with a fixed number of change points. 
 The details will be presented in a separate paper \cite{Hart14}. The general case of 
 variable speed  still offers more challenges, in spite of recent progress \cite{MZ,Mallein}.
 \end{remark}

\subsection{Outline of the proof}

The proof of our result follows the strategy used in \cite{ABK_E}.  The main difference is that we show that particles that will reach the level of the extremes at time $t$ must,  at the time of the speed change, $tb$, lie in a $\sqrt t$-neighbourhood of
 a value $\sqrt 2(\s_2-1)bt$ below the straight line of slope $\sqrt 2$. This is the done in Section 2. Then two pieces of information are needed: in Section 3 we get precise bounds on the probabilities of BBM to reach values at excessively large heights, and more generally we control the behaviour of solutions of the F-KPP equations very much ahead of the travelling wave front. The final results comes from combining this information with the precise distribution of particles at the time of the speed change. This is done in Section 4 by proving the convergence of a certain martingale, analogous, but distinct from the derivative martingale that appears in normal BBM. The identification and the proof of $L^1$ convergence of this martingale is the key idea. Using this information in Sections 5 and 6, the convergence of the maximums, respectively the Laplace functional of the extremal process are proven, much along the lines on \cite{ABK_E}. Section 7
 provides various characterisations of the limiting process, as in \cite{ABK_E}. 
 In particular, we describe the extremal process in terms of an auxiliary process, constructed from a Poisson point process with a strange intensity to those atoms we add 
 BBM's with negative drift. Interestingly, the process of the cluster extremes of this auxiliary process is again Poisson with random intensity driven by the new martingale.
 The results stated above follow then from looking at the clusters from their maximal points. 
 In the final Section 8, we give the simple proof of Theorem 
\thv(speed.2)

\noindent\textbf{Acknowledgements.} This paper uses many of the ideas from the long collaboration 
with Louis-Pierre Arguin and Nicola Kistler. We are very grateful for their input. A.B. thanks Ofer Zeitouni for discussions. Finally, we thank an anonymous referee for a very careful reading of our paper and for numerous valuable suggestions.

\section{Position of extremal particles at time $bt$}
The key to understanding the behaviour of the two speed BBM is to control  the positions of  particle at time $bt$ which are in the top at time $t$. This is done using Gaussian estimates.  

\begin{proposition}\TH(P.Prop1) Let $\s_1<\s_2$. For any $d\in \R$ and 
any  $\e>0$,  there exists a constant $A>0$ such that for all $t$ large enough
\be\Eq(p.1)
\P\left[\exists_{j\leq n(t)}\mbox{ \rm s.t. }x_j(t)>\tilde m(t)-d \mbox{ \rm  and }x_j(bt)-\sqrt{2}\s_1^2bt\not \in [ -A\sqrt{t}, A\sqrt{t}]\right]\leq \e.
\ee
\end{proposition}

\begin{proof}%%The probability in \eqv(p.1) can be written as
%\be\Eq(p.2)
%\P\left[ \#1\leq j\leq n(t)\mbox{ s.t. }x_j(t)>\tilde m(t)-d\mbox{ and }x_j(\frac{t}{2})\not \in [\sqrt{2}\s_1^2t/2-A\sqrt{t},\sqrt{2}\s_1^2t/2+A\sqrt{t}]\geq 1\right]
%\ee
Using a first order Chebyshev inequality we bound \eqv(p.1) by
\bea\Eq(p.3)
&&e^t\E\left[\1_{\left\{\s_1\sqrt{bt} w_1-\sqrt{2}\s_1^2bt\not \in[ -A\sqrt{t}, A\sqrt{t}] \right\}}
\P_{w_2}\left( \s_2\sqrt{(1-b)t} w_2> \tilde m(t)-d- \s_1\sqrt{bt} w_1
\right)\right]\nonumber\\
&=&e^t\E\Bigl[\1_{\{w_1-\sqrt{2}\s_1\sqrt{bt}\not\in[ -A', A']\}}\P_{w_2}\Bigl({w_2>\sfrac{\sqrt{2t}-\s_1\sqrt{b} w_1}{\s_2\sqrt{1-b}} -\sfrac{\log t}{2\sqrt 2\s_2\sqrt{(1-b)t} }-\sfrac{d}{\s_2\sqrt{(1-b)t}}}\Bigr)\Bigr]\nonumber\\&\equiv& (R1)+(R2),
\eea
where  $w_1,w_2$ are independent $\mathcal{N}(0,1)$-distributed, $A'=\frac{1}{\sqrt{b}\s_1}A$, $\P_{w_2}$ denotes the law of the variable $w_2$.
Introducing into the last line the identity in the form 
\be
1=\1_{\{\sqrt{2t}-\s_1\sqrt{b} w_1<\log t\}}+\1_{\{\sqrt{2t}-\s_1\sqrt{b} w_1\geq\log t\}}
\ee
 we can write it as 
$(R1)+(R2)$.

We first show $\lim_{t\to\infty}(R1)=0$. Using the standard Gaussian tail estimate
\be\Eq(sgt.1)
\int_u^\infty e^{-x^2/2}dx \leq u^{-1}e^{-u^2/2},
\ee
 $(R1)$ is bounded from above by
\be
e^t\P\left[\sqrt{2t}-\s_1\sqrt{b} w_1<\log t\right]%=e^t\P\left[w_1>\sfrac{\sqrt{2t}-\log t}{\sqrt{b}\s_1}\right]
\leq e^{t(1-1/b\s_1^2) +t^{1/2}\log t/b\s_1^2}   
\to 0\quad\mbox{as }t\to\infty.
\ee
For (R2) we can use  again \eqv(sgt.1) to show
that (R2) is smaller than
\bea\Eq(p.4)
&&e^t(2\pi)^{-1}\int_{{w -\sqrt{2}\s_1\sqrt{bt}\not\in[ -A', +A']}\atop{\sqrt{2t}-\s_1\sqrt{b} w_1\geq\log t}}
\frac{e^{-w^2/2}}{\sfrac{\sqrt{2t}}{\sqrt{1-b}\s_2} -\sfrac{\s_1\sqrt{b}}{\s_2\sqrt{1-b}}w}
\\\nonumber &&\times \exp\left(-\sfrac 12\left(\sfrac{\sqrt{2t}-\s_1\sqrt bw -\log t/(2\sqrt 2\sqrt t)-d/\sqrt t}{\sqrt{1-b}\s_2}\right)^2\right)\mbox{d}w.
\eea
We change  variables $w=\sqrt{2}\s_1\sqrt{bt}+z$. Then the integral in \eqv(p.4)
can be bounded from above by
\be\Eq(p.5)
% M e^{\left(1-\frac{2}{\s_2}-\frac{\s_1^2}{\s_2^2}+\frac{2\s_1^2}{\s_2^2}\right)}\int_{z\not\in[-A',A']} (2\pi)^{-1} e^{-z^2/\s_2^2}\mbox{d}z= 
\frac M{\sqrt{2\pi \s_2^2(1-b)}}\int_{z\not\in[-A',A']}   e^{-\frac{z^2}{2\s_2^2(1-b)}}\mbox{d}z,
\ee
where $M$ is some positive constant. \eqv(p.5) can be made as small as desired by taking $A$ (and thus $A'$) sufficiently large.
\end{proof}

\begin{remark} The point here is that since $\s_1^2<\s_1$, these particles are way below \\
$\max_{k\leq n(bt)}x_k(bt)$, which is near $\sqrt 2 \s_1 bt$. The offspring of these particles that want to be top at time will have to race much faster (at speed $\sqrt 2 \s_2^2$, rather than just $\sqrt 2 \s_2$) than normal. Fortunately, there are lots of 
particles to choose from. We will have to control precisely how many.
\end{remark}
We need a slightly finer control on the path of the extremal particle until the time of speed change. To this end 
we  define two sets on the space of paths, $X:\R_+\rightarrow \R$, The first  controls that the position of the path  is in a certain tube up to time $s$ and the second  the position of the particle at time $s$.
\bea\Eq(D.11)
\TT_{s,r}&=&\bigl\{X \big \vert  \forall_{0\leq q\leq s} | X(q)- \sfrac qsX(s)|\leq ((q\wedge (s-q))\vee r)^\g\bigr\}
\nonumber\\
\GG_{s,A,\g}&=&\bigl\{X\big \vert X(s)-\sqrt{2}\s_1^2s\in [-A  s^\g,+As^\g]\bigr\}
\eea
Recall \cite{B_M} that the ancestral path form $0$ to $x_k(s)$ can 
be written as 
$x_k(q)= \frac qsx_k(s) + \mathfrak z_k(s)$, where $\mathfrak z_k$ is a Brownian bridge from $0$ to $0$ in time $s$, independent of $x_k(s)$.
We need the following simple fact about Brownian bridges.

\begin{lemma}\TH(simplefact)
 Let $\mathfrak z(q)$ be a Brownian bridge starting in zero and ending in zero at time $s$.  Then for all $\g>1/2$, the following is true. For all $\e>0$ there exists $r$ such that
\be
\lim_{s\uparrow\infty}\P\left(\vert \mathfrak z(q) \vert < ((q\wedge (s-q))\vee r)^\g, \forall \,0\leq q\leq s \right)>1-\e.
\ee
\end{lemma}

\begin{proposition}\TH(Prop.p.4)
 Let $\s_1<\s_2$. For any $d\in \R,A>0$ ,  $\g>\frac{1}{2}$ and 
any  $\e>0$,  there exists constants $B>0$ such that, for all $t$ large enough,
\be\Eq(pp.0)
\P\left[\exists_{j\leq n(t)}: x_j(t)>\tilde m(t)-d \land x_j\in \GG_{bt,A,\frac{1}{2}} \land x_j\not\in \GG_{b\sqrt t,B,\g} \right]\leq \e.
\ee
\end{proposition}
\begin{proof}
 For $B$ and $t$ sufficiently large the probability in \eqv(pp.0) is bounded from above by 
 \be\Eq(pp.2)
 \P\left[\exists_{j\leq n(t)}: x_j(t)>\tilde m(t)-d \land x_j\in \GG_{bt,A,\frac{1}{2}} \land x_j\not\in\TT_{bt,r} \right]
 \ee
 Let $w_1$ and $w_2$ be independent $\NN(0,1)$-distributed 
 random variables and $\mathfrak z$ a Brownian bridge starting in zero and ending in zero at time $bt$. Using a first moment method as in the proof of Proposition \thv(P.Prop1) together with the independence of the Brownian bridge from its endpoint, one sees that  \eqv(pp.2) is bounded from above by 
 \bea\Eq(pp.3)
&& e^t\E\left[\1_{\left\{\s_1\sqrt{bt} w_1-\sqrt{2}\s_1^2bt  \in[ -A\sqrt{t}, A\sqrt{t}] \right\}}
\P_{w_2}\left( \s_2\sqrt{(1-b)t} w_2> \tilde m(t)-d- \s_1\sqrt{bt} w_1
\right)\right]\nonumber
\\
&&\times\P\left[\mathfrak z\not\in\TT_{bt,r}\right]<\e,
 \eea
 where the last bound follows from 
 Lemma \thv(simplefact)  (with $\e$ replaced by $\e/M$) and the bound \eqv(p.5) obtained in the proof of Proposition \thv(P.Prop1).
\end{proof}
\begin{proposition}\TH(Prop.p.5)
 Let $\s_1<\s_2$. For any $d\in \R,A,B>0$, $\g>\frac{1}{2}$ and 
any  $\e>0$,  there exists a constant $r>0$ such that for all $t$ large enough
\bea\Eq(pp.1)
&&\P\Bigl[\exists_{j\leq n(t)}: x_j(t)>\tilde m(t)-d \land x_j\in \GG_{bt,A,\frac{1}{2}} \cap\GG_{b\sqrt t,B,\g}\nonumber
\\&&\qquad\land x_j(b\sqrt t+\cdot)-x_j(b\sqrt t)\not\in\TT_{b(t-\sqrt t),r}\Bigr]\leq \e.
\eea
\end{proposition}
\begin{proof}
The proof of this proposition is essentially identical to the proof of Proposition \thv(Prop.p.4).
%  Let $w_1$ and $w_2$ be independent $\NN(0,1)$-distributed random variables and $\mathfrak z$ a Brownian bridge starting in zero and ending in zero at time $b(t-\sqrt t)$. As in the proof of Proposition \thv(Prop.p.4) we us a first moment method togethter the independence of the Brownian bridge from its endpoint to bound the probability in \eqv(pp.1) from above by  
 %\bea\Eq(pp.4)
%&& e^t\E\left[\1_{\left\{\s_1\sqrt{bt} w_1-\sqrt{2}\s_1^2bt  \in[ -A\sqrt{t}, A\sqrt{t}] \right\}}
%\P_{w_2}\left( \s_2\sqrt{(1-b)t} w_2> \tilde m(t)-d- \s_1\sqrt{bt} w_1
%\right)\right]\nonumber\\
%&&\times \P\left[\mathfrak z\not\in\TT_{b(t-\sqrt t),r}\right]<\e,
 %\eea
 %by Lemma \thv(simplefact) (with $\e$ replaced by $\e/M$) and Proposition \thv(P.Prop1).
\end{proof}

%%begin{lemma}[\cite{FZ_RW}]
% For all $\e>0$ exists $d\in \R$ such that 
%\be
%\P\left[\max_{k\leq n(t)}x_k(t)\leq m(t)-d\right]\leq \e.
%\ee
%\end{lemma}

\section{Asymptotic behaviour of BBM}
Let $\tilde x(t)$ denote a standard BBM. We are interested in the asymptotic behavior of 
\be
\P\left[\max_{1\leq i\leq n(t)} \bar {x}_i(t)>x+\sqrt{2}t\right]
\ee
for $x=at+b\sqrt{t}$, $a\in\R_+,b\in \R$.
 Recall that 
$\P\left(\max_{k\leq n(t)}\bar x_k(t)>x\right)$ is the solution of the 
F-KPP equation 
\be\Eq(fkpp)
\del_t u(t,x)=\frac 12 \del_x^2u(t,x) +(1-u(t,x))-\sum_{k=1}^\infty p_k (1-u(t,x))^k.
\ee
with initial condition $u(0,x)=\1_{x<0}$. We are more generally interested in the behaviour of solutions for such large values of $x$.
The following proposition is an extension of Lemma 4.5 in \cite{ABK_E} for these values of $x$.

\begin{proposition}\TH(A.Prop1)
Let $u$ be a solution to the F-KPP equation with initial data satisfying \\
(i) $0\leq u(0,x)\leq 1$;\\
(ii) for some $h>0$, $\limsup_{t\to\infty}\frac{1}{t}\log\int_t^{t(1+h)}u(0,y)\mbox{d}y\leq-\sqrt{2}$;\\
(iii) for some $v>0,\,M>0,\,N>0$, it holds that $\int_{x}^{x+N}u(0,y)\mbox{d}y>v$ for all $x\leq -M$;\\
(iv) moreover, $\int_0^\infty u(0,y)ye^{2y}\mbox{d}y<\infty$.\\
Then we have for $x=at+o(t)$
\be
\lim_{t\to\infty}e^{\sqrt{2}x}e^{x^2/2t}t^{1/2}u(t,x+\sqrt{2}t)=C(a),
\ee
where $C(a)$ is a strictly positive constant. The convergence is uniform for $a$ in  compact intervals.
\end{proposition}

Define for $r>0$ the function $\Psi(r,t,x+\sqrt{2}t)$ by
\bea\Eq(A.1)
&&\Psi(r,t,x+\sqrt{2}t)
=\\\nonumber
&&\frac{e^{-\sqrt{2}x}}{\sqrt{2\pi(t-r)}}\int_0^\infty u(r,y+\sqrt{2}r)e^{\sqrt{2}y}e^{-\frac{(y-x)^2}{2(t-r)}}\left[1-e^{-2y\left(\frac{x+\frac{3}{2\sqrt{2}}\log t}{t-r}\right)}\right]\mbox{d}y.
\eea
\begin{lemma}\TH(fz.6) For $x=at+o(t)$ we have, under the assumptions of Proposition \thv(A.Prop1), 
\begin{eqnarray}
&&\lim_{t\to\infty}e^{\sqrt{2}x}e^{x^2/2t}t^{1/2}\Psi(r,t,x+\sqrt{2}t)
\\\nonumber
&&=\frac{1}{\sqrt{2\pi}}\int_0^\infty e^{-a^2r/2} u(r,y+\sqrt{2}r)e^{(\sqrt{2}+a)y}\left(1-e^{-2ay}\right)\mbox{d}y\equiv C(r,a).
\end{eqnarray}
 The convergence is uniform for $a$ in a compact set.
\end{lemma}

\begin{proof}
Using  \eqv(A.1) we have
\bea\Eq(A.2)
&&\lim_{t\to\infty}e^{\sqrt{2}x}e^{x^2/2t}t^{1/2}\Psi(r,t,x+\sqrt{2}t)\nonumber\\\nonumber
&=& \lim_{t\to\infty}\frac{\sqrt{t}}{\sqrt{2\pi(t-r)}}e^{x^2/2t}\int_0^\infty u(r,y+\sqrt{2}r)e^{\sqrt{2}y}e^{-\frac{(y-x)^2}{2(t-r)}}\\
&&\quad\times\left[1-\exp\left(-2y\left(\frac{x+\sfrac{3}{2\sqrt{2}}\log t}{t-r}\right)\right)\right]\mbox{d}y .
\eea
Next we show that we can use dominated convergence to take the limit $t\to\infty$ into the integral. 
First, the integrand is bounded by
\be
B e^{-a^2r/2}u(r,y+\sqrt{2}r)e^{(\sqrt{2}+a+1)y},
\ee
where $B>0$. As was shown by Bramson \cite{B_C} (and used in \cite{ABK_E}),  the solution of the F-KPP equation can be bounded by the solution  $u^{(2)}(t,x)$  of the linearised F-KPP equation
\be
\partial_t u^{(2)}=\frac{1}{2}u_{xx}^{(2)}-u^{(2)}
\ee
with the same initial condition 
$u^{(2)}(0,x)=u(0,x)$.
%By the maximum principle for nonlinear parabolic pde's we have 
%\be\Eq(A.3)
%u(t,x)\leq u^{(2)}(t,x).
%\ee
Moreover there exists $y_0$ such that for any $x>0$
\be\Eq(A.4)
u^{(2)}(t,x)\leq e^te^{-x^2/2t}e^{y_0x/t}
\ee
Thus we get that 
\bea
& &\int_0^\infty B e^{-a^2r/2}u(r,y+\sqrt{2}r)e^{(\sqrt{2}+a+1)y}\mbox{d}y\nonumber\\
&\leq &\int_0^\infty B(r) e^{-a^2r/2}e^{-y^2/2r}e^{(a+1)y}\mbox{dy}<\infty.
\eea
where $B(r)$ is a constant that only depends on $r$. Hence we can apply dominated convergence to \eqv(A.2) and obtain
\bea\Eq(A.4.1)
& & \frac{1}{\sqrt{2\pi}}\int_0^\infty u(r,y+\sqrt{2}r)e^{\sqrt{2}y}\lim_{t\to \infty}\left[e^{\sqrt{2}y}e^{-\frac{(y-x)^2}{2(t-r)}}\left[1-e^{-2y\left(\frac{x+\frac{3}{2\sqrt{2}}\log t}{t-r}\right)}\right]\right]\mbox{d}y\nonumber\\
&=&\frac{1}{\sqrt{2\pi}}\int_0^\infty e^{-a^2r/2} u(r,y+\sqrt{2}r)e^{(\sqrt{2}+a)y}\left(1-e^{-2ay}\right)\mbox{d}y.
\eea
 This proves the lemma.
\end{proof}

\begin{proof}[Proof of Proposition \thv(A.Prop1)]
Due to the assumptions (i),(ii),(iii) and (iv) we can use Proposition 4.3 of \cite{ABK_E} for $t>8r$ and $x>8r-\frac{3}{2\sqrt{2}}\log t$ and $r$ large enough:
\be\Eq(A.4.2)
\gamma^{-1}(r) \Psi(r,t,x+\sqrt{2}t)
\leq u(t,x+\sqrt{2}t)
\leq\gamma(r) \Psi(r,t,x+\sqrt{2}t) ,
\ee
where $\g(r)$ does not depend $x$ and $t$ and $\lim_{r\to\infty}\g(r)=1$. %%%limsup_{t\to\infty}e^{\sqrt{2}x}e^{x^2/2t}t^{1/2}u(t,x+\sqrt{2}t)&\leq& \gamma(r)\lim_{t\to\infty}e^{\sqrt{2}x}e^{x^2/2t}t^{1/2}\Psi(r,t,x+\sqrt{2}t)\\&=&\gamma(r)C(r,a)
%\eel
%and
%\bea\nonumber
%\liminf_{t\to\infty}e^{\sqrt{2}x}e^{x^2/2t}t^{1/2}u(t,x+\sqrt{2}t)&\geq &\gamma(r)^{-1}\lim_{t\to\infty}e^{\sqrt{2}x}e^{x^2/2t}t^{1/2}\Psi(r,t,x+\sqrt{2}t)\\&=&\gamma(r)^{-1}C(r,a).
%\eea
Since $\g(r)\to 1$ as $r\to\infty$ this implies
\be
\limsup_{t\to\infty}e^{\sqrt{2}x}e^{x^2/2t}t^{1/2}u(t,x+\sqrt{2}t)\leq \liminf_{r\to\infty}C(r,a)
\ee
and
\be
\liminf_{t\to\infty}e^{\sqrt{2}x}e^{x^2/2t}t^{1/2}u(t,x+\sqrt{2}t)\geq \limsup_{r\to\infty}C(r,a)
\ee
Hence $\lim_{r\to\infty}C(r,a)=C(a)$ exists. Moreover, 
\be
\lim_{t\to\infty}e^{\sqrt{2}x}e^{x^2/2t}t^{1/2}u(t,x+\sqrt{2}t)
\ee
 exists and is equal to $C(a)$. It is left to show that $C(a) \neq 0$ for $a>0$. If $C(a)=0$ we would have
\be \Eq(A.5)
\lim_{t\to\infty}e^{\sqrt{2}x}e^{x^2/2t}t^{1/2}u(t,x+\sqrt{2}t)=0,
\ee
but 
\be\lim_{t\to\infty}e^{\sqrt{2}x}e^{x^2/2t}t^{1/2}u(t,x+\sqrt{2}t)\geq C(r,a)\g(r)^{-1},
\ee
for $r$ large enough, by \eqv(A.4.2). This contradicts \eqv(A.5). The same proposition implies 
 \be
\lim_{t\to\infty}e^{\sqrt{2}x}e^{x^2/2t}t^{1/2}u(t,x+\sqrt{2}t)\leq C(r,a)\g(r).
 \ee
 Hence $C(a)\neq \infty$.   Proposition \thv(A.Prop1) is proven. 
\end{proof}

\section{The McKean martingale}

In this section we pick up the idea of \cite{LS} and consider a suitable convergent martingale for the time inhomogeneous BBM with $\s_1<\s_2$. Let $x_i(s),\,1\leq i\leq n(s)$ be the particles of a BBM where the Brownian motions have variance $\s_1^2$ with $\s_1^2<1$.  Define
\be\Eq(D.def)
Y_s=\sum_{i=1}^{n(s)}e^{-s(1+\s_1^2)+\sqrt{2}x_i(s)}.
\ee
This turns out to be a uniformly integrable martingale that converges almost surely to a positive limit $Y$. 

\begin{remark}
Note that in terms of statistical mechanics, 
$Y_s$ can be thought of as a normalised partition function at  inverse temperature
$\s_1\sqrt 2$ (for ordinary BBM). Here the critical temperature is $\sqrt 2$, so that we are in the high-temperature case. 
In the case of the GREM, where the underlying tree is deterministic, this quantity is known to even converge to a constant \cite{BK1}.  
\end{remark}

\begin{theorem}\TH(D.Th1)
The limit $\lim_{s\to\infty}Y_s$ exists almost surely and in $L^1$, is finite and strictly positive.
\end{theorem}

The assertion of this theorem follows immediately from the following proposition.

\begin{proposition}\TH(D.Prop1) If $\s_1<1$, 
$Y_s$ is a uniformly integrable martingale with $\E[Y_s]=1$
\end{proposition}

\begin{remark} We would like to call this martingale \emph{McKean martingale}, since McKean \cite{McKean} had originally conjectured that this martingale 
(with $\s_1=1$) was the martingale in the representation of the extremal distribution of BBM, which, as Lalley and Sellke showed is wrong as it is actually the derivative martingale that appears there. We find it nice to see that in the time-inhomogeneous case with $\s_1<1$, 
KcKean turns out to be right! We will see in the proof that the uniform integrability of this martingale breaks down at $\s_1=1$.
\end{remark}

\begin{remark}Note further that if $\s_1=0$, then $Y_t=e^{-t} n(t)$ which is well known to converge to an exponential random variable.
\end{remark}

\begin{proof}
Clearly,
\be\Eq(D.1)
\E [Y_s]=e^s \E\left[e^{-(1+\s_1^2)s+\sqrt{2}x_1(s)}\right]=1.
\ee
Next we show that $Y_s$ is a martingale. Let $0<r<s$. Then
\be\Eq(D.2)
\E[Y_s\vert \FF_r]=\sum_{i=1}^{n(r)}\E\left[\sum_{j=1}^{n_j(s-r)}e^{-s(1+\s_1^2)+\sqrt{2}(x_j^i(s-r)+x_i(r))}\left \vert \right.\FF_r\right],
\ee
where for $1\leq i\leq r$: $\left\{x_j^i(s-r),1\leq j\leq n_i(s-r)\right\}$ are the particles of independent BBM's with variance $\s_1^2$ at time $s-r$. \eqv(D.2) is equal to
\be
\sum_{i=1}^{n(r)}e^{-r(1+\s_1^2)+\sqrt{2}x_i(r)}=Y_r,
\ee
as desired.

 It remains to show that $Y_s$ is uniformly integrable. 
% To do so we use a truncated second moment method. We  define two sets on the space of paths, $X:\R_+\rightarrow \R$, The first  controls that the position of the path  is in a certain tube up to time $s$ and the second  the position of the particle at time $s$.
% \bea\Eq(D.11)
% \TT_{s,r}&=&\bigl\{X \big \vert  \forall_{0\leq q\leq s} | X(q)- \sfrac qsX(s)|\leq ((q\wedge (s-q))\vee r)^\g\bigr\}
% \nonumber\\
% \GG_{s,A,\g}&=&\bigl\{X\big \vert X(s)-\sqrt{2}\s_1^2s\in [-A  s^\g,+As^\g]\bigr\}
% \eea
We will  write abusively $x_k(r)$ for the ancestor of $x_k(s)$ at time $r\leq s$
and write $x_k$ for the entire ancestral path of $x_k(s)$.
Define the truncated variable 
\be\Eq(D.6)
Y_s^{A}=\sum_{i=1}^{n(s)}e^{-(1+\sigma_1^2)s+\sqrt{2}x_i(s)}\1_{\{x_i\in \GG_{s,A,1/2},x_i\in \TT_{s,r}\}}.
\ee

% Recall \cite{B_M} that the ancestral path form $0$ to $x_k(s)$ can 
% be written as 
% $x_k(q)= \frac qsx_k(s) + \mathfrak z_k(s)$, where $\mathfrak z_k$ is a Brownian bridge from $0$ to $0$ in time $s$, independent of $x_k(s)$.
% We need the following simple fact about Brownian bridges.
% 
% \begin{lemma}\TH(simplefact)
%  Let $\mathfrak z(q)$ be a Brownian bridge starting in zero and ending in zero at time $s$.  Then for all $\g>1/2$, the following is true. For all $\e>0$ there exits $r$ such that
% \be
% \lim_{s\uparrow\infty}\P\left(\vert \mathfrak z(q) \vert < ((q\wedge (s-q))\vee r)^\g, \forall \,0\leq q\leq s \right)>1-\e.
% \ee
% \end{lemma}

First $Y_s-Y_s^A\geq 0$, and a simple computation using the independence of $x_k(s)$ and $x_k(r)- \sfrac rsx_k(s)$  together with Lemma 
\thv(simplefact) shows that 
%sum_{i=1}^{n(s)}e^{-(1+\sigma_1^2)s+\sqrt{2}x_i(s)}\1_{\{x_i(s)<\sqrt{2}\s_1^2s-A\sqrt{s}\}}.
\bea\Eq(D.3)\nonumber
\E\left[Y_s-Y_s^{A}\right]&\leq& e^s \int_{-\infty}^{\infty}e^{-(1+\s_1^2)s+\sqrt{2s}\s_1x}\1_{\{x-\sqrt{2s}\s_1\not\in [-A,A]\}}e^{-x^2/2}\sfrac{{d}x}{\sqrt{2\pi}} +\e\\
&=& \int_{|z|>A}e^{-z^2/2}\sfrac{{d}z}{\sqrt{2\pi}} +\e,
\eea
which can be made as small as desired by taking $A$ and $r$ to infinity.
The key point is that the the second moment of $Y^A_s$ is uniformly bounded in $s$. 
\be
\E\left[(Y_s^A)^2\right]=\E\Biggl[\Biggl(\sum_{i=1}^{n(s)}e^{-(1+\sigma_1^2)s+\sqrt{2}x_i(s)}\1_{\{x_i\in \GG_{s,A,1/2}\cap \TT_{s,r}\}}\Biggr)^2\Biggr] 
\equiv (T1)+(T2),
\ee
where
\bea\Eq(D.4)
(T1)&=&\E\Biggl[\sum_{i=1}^{n(s)}e^{-2((1+\sigma_1^2)s-\sqrt{2}x_i(s))}\1_{\{x_i\in \GG_{s,A,1/2}\cap \TT_{s,r}\}}\Biggr]\nonumber\\
(T2)&=&\E\biggl[\sum_{\stackrel{i,j=1}{i\neq j}}e^{-2(1+\sigma_1^2)s+\sqrt{2}(x_i(s)+x_j(s))}\1_{\{x_i,x_j\in \GG_{s,A,1/2}\cap \TT_{s,r}\}}\biggr]
\eea
We start by controlling $(T1)$. 
\bea
(T1)\leq \frac{e^{(s-2s(1+\s_1^2)}}{\sqrt{2\pi}}\int_{\sqrt{2s}\s_1-A/\s_1}^{\sqrt{2s}\s_1+A/\s_1}e^{2\sqrt{2s}\s_1x}e^{-x^2/2}{{d}x}\nonumber\\
=\frac{e^{-(1-\s_1^2)s}}{\sqrt {2\pi}}\int_{-A/\s_1}^{A/\s_1}e^{-x^2/2}dx\leq 
 e^{-(1-\s_1^2)s}\to 0 \quad \mbox{as }s\to\infty.
\eea
Now we control $(T2)$. By the sometimes so-called "many-to-two lemma" (see e.g.\cite{B_C}, Lemma 10), and dropping the useless parts of the conditions on the Brownian bridges
\begin{eqnarray}
\Eq(D.10)
(T2)\leq &&Ke^s\int_0^{s}e^{s-q}\int_{\sqrt{2}\s_1^2{q}-
I_1(q,s)}^{\sqrt{2}\s_1^2{q}+I_1(q,s)}
\Biggl(\int_{\sqrt{2}\s_1^2s-A\sqrt{s}-x}^{\sqrt{2}\s_1^2s+A\sqrt{s}-x}e^{-s(1+\s_1^2)+\sqrt{2}(x+y)} \nonumber\\ 
&&\quad\quad\quad \times e^{-\frac{y^2}{2\s_1^2(s-q)}}\sfrac{{d}y}{\s_1\sqrt{2\pi(s-q)}}\Biggr)^2 e^{-\frac{x^2}{2q\s_1^2}}\sfrac{{d}xdq}{\sqrt{2\pi\s_1^2q}} ,
\end{eqnarray}
where $K=\sum_{i=1}^{\infty}p_kk(k-1)$ and $I_1(q,s)={Aq/\sqrt{s}+((q\wedge (s-q))\lor r)^\g}$. Moreover
 We   change  variables $x=z+\sqrt{2}\s_1^2q$ and obtain
\begin{eqnarray}\Eq(D.6)
&& Ke^s\int_0^{s}e^{s-q}\int_{ -I_1(q,s)}^{+I_1(q,s)}
\Biggl(\int_{\sqrt{2}\s_1^2(s-q)-A\sqrt{s}-z}^{\sqrt{2}\s_1^2(s-q)+A\sqrt{s}-z}e^{-s(1+\s_1^2)+\sqrt{2}(z+\sqrt{2}\s_1^2q+y)} \nonumber\\ 
&&\quad\quad \times e^{-\frac{y^2}{2\s_1^2(s-q)}}\sfrac{{d}y}{\s_1\sqrt{2\pi(s-q)}}\Biggr)^2 e^{-\frac{(z+\sqrt{2}\s_1^2{q})^2}{2\s_1^2q}}\sfrac{{d}zdq}{\sqrt{2\pi\s_1^2q}},
\end{eqnarray}
Now we change  variables $w=\frac{y}{\s_1\sqrt{s-q}}-\sqrt{2}\s_1\sqrt{s-q}$. \eqv(D.6) is equal to
 %\begin{eqnarray}
%&& Ke^s\int_0^{s}e^{s-q}\int_{ -I_1(q,s)}^{ +I_1(q,s)}
%\Biggl(\int_{\frac{-A\sqrt{s}-z}{\s_1\sqrt{s-q}}}^{ \frac{+A\sqrt{s}-z}{\s_1\sqrt{s-q}}}e^{-s(1-\s_1)+\sqrt{2}(z+ w\s_1\sqrt{s-q})} \nonumber\\ 
%&&\quad\quad \times e^{-\frac{(w+\sqrt{2}\s_1\sqrt{q})^2}{2}}\frac{\mbox{d}w}{\sqrt{2\pi}}\Biggr)^2 %e^{-\frac{(z+\sqrt{2}\s_1^2{q})^2}{2\s_1^2}q}\frac{\mbox{d}z}{\sqrt{2\pi\s_1^2q}} dq .
%\end{eqnarray}
%This is equal to 
%\begin{eqnarray}\Eq(D.7)
\be\Eq(D.7)
 K\int_0^{s}e^{-q(1-2\s_1^2)}\int_{ -I_1(q,s)}^{+I_1(q,s)}e^{+2\sqrt{2}z} 
\Biggl(\int_{\frac{-A\sqrt{s}-z}{\s_1\sqrt{s-q}}}^{ \frac{+A\sqrt{s}-z}{\s_1\sqrt{s-q}}}
 e^{-{w^2}/{2}}\sfrac{{d}w}{\sqrt{2\pi}}\Biggr)^2 e^{-\frac{(z+\sqrt{2}\s_1^2{q})^2}{2\s_1^2q}}\sfrac{{d}zdq}{\sqrt{2\pi\s_1^2q}}.
%\end{eqnarray}
\ee
Now the integral with respect to $w$ is bounded by $1$. Hence \eqv(D.7) is bounded from above by
\be\Eq(D.9)
 K\int_0^{s}e^{-q(1-2\s_1^2)}\int_{ -I_1(q,s)/\s_1\sqrt q}^{+I_1(q,s)/\s_1\sqrt q} 
e^{-\frac{(z-\sqrt{2}\s_1\sqrt{q})^2}{2}}\sfrac{{d}zdq}{\sqrt{2\pi}}.
%\end{eqnarray}
\ee
We split the integral over $q$ into the three parts $R_1$, $R_2$, and $R_3$ according to the 
integration from $0$ to $r$, $r$ to $s-r$, and $s-r$ to $r$, respectively.
% \be\Eq(D.8)
%   Ke^s\int_r^{s-r}e^{s-q}\int_{ -I_1(q,s)}^{+I_1(q,s)}
% \Biggl( e^{-s(1-\s_1^2)+\sqrt{2}(z\s_1\sqrt{q})} 
% %\nonumber\\ 
% %&&\quad\quad \times  
% e^{-\s_1^2(s-q)} \Biggr)^2 e^{-(z+\sqrt{2}\s_1\sqrt{q})^2/2}\frac{\mbox{d}z}{\sqrt{2\pi}} dq .
% \ee
% We rewrite \eqv(D.8) as
Then 
\be
R_2\leq K\int_r^{s-r}
e^{-q(1-2\s_1^2)}
 \frac{e^{-\frac 12\left(I_1(q,s)/\s_1\sqrt q-\sqrt{2}\s_1\sqrt{q}\right)^2}}{\sqrt{2\pi}\left(\sqrt{2}\s_1\sqrt{q}-I_1(q,s)/\s_1\sqrt q\right)}{dq}
\ee
This is bounded by
\be
K\int_r^{s-r}e^{-(1-\s_1^2)q+O(q^\g)}{d}q
\leq \frac {C}{1-\s_1^2}e^{-c(1-\s_1^2)r}.
\ee
For $R_1$ the integral over $z$ can only be bounded by one. This gives 
\be
R_1\leq  K\int_0^re^{(2\s_1^2-1)q}{d}q\equiv D_1(r),
\ee
$R_3$ can be treated the same way as $R_2$ and 
we get 
\be
R_3
\leq K\int_{s-r}^s e^{-(1-\s_1^2)q+O(r^\g)}{d}q\leq\frac{K}{1-\s_1^2}e^{-(1-\s_1^2)(s-r)+O(r^\g)}\to 0 \quad\mbox{as }s\to\infty.
\ee
Putting all three estimates together, we see that  
$\sup_{s}\E\left[ \left(Y^A_s\right)^2\right]\leq D_2(r)$.
From this it follows that $Y_s$ is uniformly integrable. Namely,
\bea\Eq(ui.1)
\E [Y_s\1_{Y_s>z}] &=&\E[ Y_s^A \1_{Y_s>z}] + \E [ (Y_s-Y_s^A) \1_{Y_s>z}] 
\\\nonumber
&=&\E\left[ Y_s^A \1_{Y^A_s>z/2}\right] +\E\left[ Y_s^A \left(\1_{Y_s>z}-\1_{Y^A_s>z/2}\right)\right]+ \E [ (Y_s-Y_s^A) \1_{Y_s>z}].
\eea
 For the first term we have
 \be\Eq(ui.2)
\E\left[ Y_s^A \1_{Y^A_s>z/2}\right] \leq \frac 2z \E\left[ \left(Y^A_s\right)^2\right]\leq \frac 2zD_2(r).
\ee
For the second, we have 
\bea\Eq(ui.3)
\E\left[ Y_s^A \left(\1_{Y_s>z}-\1_{Y^A_s>z/2}\right)\right]&\leq& \E\left[ Y_s^A \1_{Y_s-Y^A_s\geq z/2} \1_{Y^A_s\leq z/2}\right]
\\\nonumber &
\leq& \frac z2
\P\left[(Y_s-Y_s^A) >z/2\right] \leq \E \left[Y_s-Y^A_s\right].
\eea
The last term in \eqv(ui.1) is also bounded by $\E \left[Y_s-Y^A_s\right]$. Choosing now $A$ and $r$ such that $\E \left[Y_s-Y^A_s\right]\leq \e/3$, and then 
$z$ so large that $ \frac 2z D_2(r)\leq \e/3$, 
we obtain that $\E [Y_s\1_{Y_s>z}] \leq \e$, for large enough $z$, uniformly in $s$.
Thus  $Y_s$ is uniformly integrable, which we wanted to show. 
\end{proof}

\begin{proof}[Proof of Theorem \thv(D.Th1)]By Proposition \thv(D.Prop1) $Y_s$ is a positive uniformly integrable martingale. By Doob's martingale convergence theorem we have that $\lim Y_s=Y$ exists almost surely and is finite. Moreover $Y$ is positive and $Y_s\stackrel{L^1}{\to}Y$. In particular, this implies $Y \not\equiv 0$.
% Since $Var(Y_s)$ does not vanish in the limit $s \to \infty$ we can conclude that $Y\not \equiv 0$.
\end{proof}

We will also need to control the processes 
 $\tilde Y_{s,\g}^{A}=\sum_{i=1}^{n(s)}e^{-(1+\s_1^2)s+\sqrt{2}x_i(s)}\1_{x_i\in\GG_{s,A,\g}}$. 
 
 \begin{lemma}\TH(tilde.1)
 The family of random variables $\tilde Y_{s,\g}^{A}$, $s,A\in \R_+,1>\g>1/2$ is 
 uniformly integrable and converges, as $s\uparrow \infty$ and $A\uparrow\infty$, to $Y$, both in probability and in $L^1$.
 \end{lemma}
 
 \begin{proof} The proof of uniform integrability is a rerun of the proof of Proposition \thv(D.Prop1), noting that the bounds on the truncated second moments are uniform in $A$. Moreover, the same computation as in Eq. \eqv(D.3) shows that 
 $\E |Y_s-\tilde Y^A_{s,\g}| \leq \e$, uniformly in $s$, for $A$ large enough.
 Therefore,
 \be\Eq(tilde.2)
 \lim_{A\uparrow\infty}\limsup_{s\uparrow \infty} \E  |Y_s-\tilde Y^A_{s,\g}| =0,
 \ee
  which implies that $Y_s-\tilde Y^A_{s,\g}$ converges to zero in probability. Since 
  $Y_s$ converges to $Y$ almost surely, we arrive at the second assertion of the lemma.
  \end{proof}
 
 \section{Convergence of the maximum of two-speed BBM}
Using the results established in the last three sections, 
we show now the convergence of the law the of the maximum of 
two-speed BBM in the case $\s_1<\s_2$. 

\begin{theorem}\TH(M.Th1)
Let $\{x_k(t),1\leq k\leq n(t)\}$ be the particles of a time inhomogeneous BBM with $\s_1<\s_2$ and the normalising assumption $\s_1^2b+\s_2^2(1-b)=1$. Then, with $\tilde m(t)$ as in Theorem \thv(speed.1),  
\be\Eq(M.Th2)
\lim_{t\to\infty}\P\left[\max_{1\leq k \leq n(t)}x_k(t)-\tilde m(t)\leq y\right]= \E\left[\exp\left(-\s_2C(a)Ye^{-\sqrt{2}y}\right)\right].
\ee
$Y$ is  the limit of the McKean martingale from the last section, 
and $C(a)$ is the positive constant given by
\be\Eq(constant.1)
C(a)=\lim_{r\to\infty}\int_0^\infty e^{-a^2r/2}\P\left[\max_{k\leq n(t)}\bar{x}_k(r)> z+\sqrt{2}r\right]e^{(\sqrt{2}+a)z}\left(1-e^{-2az}\right)\mbox{d}z,
\ee
where $\{\bar{x}_k(t),k\leq n(t)\}$ are the particles of a standard BBM 
 and $a=\sqrt{2}(\s_2-1)$.

\end{theorem}
 
\begin{proof}
Denote by $\{x_i(bt),1\leq i \leq n(bt)\}$ the particles of a BBM with variance $\s_1$ at time $bt$ and by $\FF_{bt}$ the $\s$-algebra generated this BBM. Moreover,  for $1\leq i\leq n(bt)$, let  $\{x_j^i((1-b)t),1\leq j \leq n_i((1-b)t)\}$ denote the particles of independent BBM with variance $\s_2$ at time $(1-b)t$.

Let us first observe that by  the analog of Theorem 1.1. of \cite{FZ_RW} for two-speed BBM\footnote{As pointed out in \cite{FZ_BM},  the arguments used  for branching random walks carry all over to BBM.}
we know that the maximum of our process is not too small, namely that 
for any $\e>0$, there exists $d<\infty$, such that 
\be\Eq(fz.1)
\P\left[\max_{1\leq k \leq n(t)}x_k(t)-\tilde m(t)\leq -d\right]\leq \e/2.
\ee
Therefore, 
\bea
\Eq(fz.2)
\P\left[ -d\leq \max_{1\leq k \leq n(t)}x_k(t)-\tilde m(t)\leq y\right]&\leq& 
\P\left[\max_{1\leq k \leq n(t)}x_k(t)-\tilde m(t)\leq y\right]\\\nonumber
&\leq&
\P\left[-d\leq \max_{1\leq k \leq n(t)}x_k(t)-\tilde m(t)\leq y\right]
\\\nonumber
&+&
\P\left[-d\leq \max_{1\leq k \leq n(t)}x_k(t)-\tilde m(t)\leq y\right]
+\e/2
\eea
On the other hand, by Proposition \thv(P.Prop1), we have that there exists $A<\infty$, such  that 
\bea
\Eq(fz.4)
&&\P\left[\forall_{1\leq k \leq n(t)}\{-d\leq x_k(t)-\tilde m(t)\leq y
\}\cap\{x_k\in\GG_{bt,A,\frac 12}\}\right]\\\nonumber
&&\leq
\P\left[ -d\leq \max_{1\leq k \leq n(t)}x_k(t)-\tilde m(t)\leq y\right]
\\\nonumber&&=\P\left[\forall_{1\leq k \leq n(t)}\{-d\leq x_k(t)-\tilde m(t)\leq y
\}\cap\{x_k\in\GG_{bt,A,\frac 12}\}\ \right]\\\nonumber
&&+\P\left[\forall_{1\leq k \leq n(t)}\{-d\leq x_k(t)-\tilde m(t)\leq y
\}\cap\{x_k\in\GG_{bt,A,\frac 12}\}\right]\\\nonumber
&&\leq\P\left[\forall_{1\leq k \leq n(t)}\{-d\leq x_k(t)-\tilde m(t)\leq y
\}\cap\{x_k\in\GG_{bt,A,\frac 12}\}\right]+\e/2
\eea
Combining \eqv(fz.2) and \eqv(fz.4), we have that
\bea\Eq(fz.5)
&&\P\left[\forall_{1\leq k \leq n(t)}\{-d\leq x_k(t)-\tilde m(t)\leq y
\}\cap\{x_k\in\GG_{bt,A,\frac 12}\}\right]
\nonumber\\
&& \leq \P\left[\forall_{1\leq k \leq n(t)}\{-d\leq x_k(t)-\tilde m(t)\leq y
\}\right]
\\\nonumber &&\leq
\P\left[\forall_{1\leq k \leq n(t)}\{-d\leq x_k(t)-\tilde m(t)\leq y
\}\cap\{x_k\in\GG_{bt,A,\frac 12}\}\right]+\e
\eea 
Thus we obtain 
\bea\Eq(M.1)\nonumber
&&\P\left[\max_{1\leq k \leq n(t)}x_k(t)-\tilde m(t)\leq y\right]\\
&&= \P\left[\max_{1\leq i \leq n(bt)}\max_{1\leq j \leq n_i((1-b)t)}x_i(bt)+x_j^i((1-b)t)-\tilde m(t)\leq y\right]\nonumber\\
&&=\E\left[\prod_{1\leq i\leq n_i(bt) }\P\left[\max_{1\leq j \leq n_i((1-b)t)}x_j^i((1-b)t)\leq \tilde m(t)-x_i(bt)+y \left \vert \right. \FF_{bt}\right]\right]\nonumber\\\nonumber
&&\leq\E\Biggl[\prod_{\stackrel{1\leq i\leq n(bt) }{x_i\in\GG_{bt,A,\frac 12}}}
%\nonumber\\ &&\quad\times
\P\left[\max_{1\leq j \leq n_i((1-b)t)}\s_2^{-1}x_j^i((1-b)t) \leq \s_2^{-1}\left(\tilde m(t)-x_i(tb)+y \right)\left \vert \right.  \FF_{tb}\right]\Biggr]\\&&\quad +\e.
\eea
Of course the corresponding lower bound holds without the $\e$.

%Using this notation the indicator function in \eqv(M.1) can be written as $\1_{x_i\in\GG_{tb,A}}$. 
Observe that the last  probability in \eqv(M.1)
is equal to 
\be\Eq(M.2)
1-
\P\left[\max_{1\leq j \leq n_i((1-b)t)} \bar{x}_j^i((1-b)t) > \s_2^{-1}\left(\tilde m(t)-x_i(tb)+y \right)\left \vert \right.  \FF_{tb}\right],
\ee
 where $\bar{x}_j^i((1-b)t)$ are the particles of a standard BBM. Using Proposition \thv(A.Prop1) for $(1-b)t$ and $u(t,x)=\P\left(\max \bar{x}_j^i(t)>x\right)$, and setting 
\be\Eq(approx.1)
C_t(x)\equiv e^{\sqrt 2 x+x^2/2t}t^{1/2} u(t,x+\sqrt 2 t),
\ee
we can write the probabilities  in the last line of \eqv(M.2)
as
\bea
\Eq(approx.2)
&&u\left((1-b)t, \s_2^{-1}(\tilde m(t) -x_i(bt)+y)\right)\\\nonumber
&&=C_{(1-b)t} \left(\s_2^{-1}(\tilde m(t) -x_i(bt)+y)-t\sqrt 2(1-b))\right)
\\\nonumber &&\quad\times e^{-\sqrt{2}\left(\frac{\tilde m(t)-x_i(bt)+y}{\s_2}-\sqrt{2}(1-b)t\right)}
e^{-\frac{1}{2(1-b)t}\left(\frac{\tilde m(t)-x_i(bt)+y}{\s_2}-\sqrt{2}(1-b)t\right)^2}((1-b)t)^{-1/2}
\eea
Now all the $x_i(bt)$ that appear are of the form 
$x_i(bt) =\sqrt{2}\s_1^2b t + O(\sqrt t)$,
so that 
\be\Eq(approx.3)
C_{(1-b)t} \left(\s_2^{-1}(\tilde m(t) -x_i(bt)+y)-\sqrt{2}(1-b)t)\right)
=C_{(1-b)t}( a(1-b)t +O(\sqrt t)),
\ee 
with (using \eqv(norm.1))
\be \Eq(constant.a)
a\equiv \frac{1}{1-b}\left(\frac{\sqrt {2} -\sqrt{2}\s_1^2b  }{\s_2}-\sqrt{2}(1-b)\right)=\sqrt{2}(\s_2-1),
\ee
But then,  by Proposition \thv(A.Prop1), 
\be\Eq(approx.4)
\lim_{t\uparrow\infty} C_{(1-b)t} \left(\s_2^{-1}(\tilde m(t) -x_i(bt)+y)-\sqrt{2}(1-b)t)\right)
=C(a), 
\ee
with uniform convergence for all $i$ appearing in \eqv(M.1)  and  $C(a)$ is the constant 
given by \eqv(constant.1).
Thus we can rewrite the expectation in \eqv(M.1) as
\bea\Eq(M.2.1)\nonumber
&&\E\Biggl[\prod_{\stackrel{1\leq i\leq n(bt) }{x_i\in \GG_{bt,A,1/2}}}
\P\left[\max_{1\leq j \leq n_i((1-b)t)}\s_2^{-1}x_j^i((1-b)t) \leq \s_2^{-1}\left(\tilde m(t)-x_i(tb)+y \right)\left \vert \right.  \FF_{tb}\right]\Biggr]
\\\nonumber
&&=\E\Biggl[\prod_{\stackrel{1\leq i\leq n(bt) }{x_i\in\GG_{bt,A,1/2}}}
\Bigl\{1-C(a)
e^{-\sqrt{2}\left(\sfrac{\tilde m(t)-x_i(bt)+y}{\s_2}-\sqrt{2}(1-b)t\right)}\nonumber\\
& &\qquad \times
 e^{-\frac{1}{2(1-b)t}\left(\sfrac{\tilde m(t)-x_i(bt)+y}{\s_2}-\sqrt{2}(1-b)t\right)^2}((1-b)t)^{-1/2}(1+o(1))\Bigr\}\Biggr].
 \eea

%& & \e +\E\Biggl[\prod_{\stackrel{1\leq i\leq n(b\sqrt t)}{x_i\in\GG_{b\sqrt t,B,\g}} } \E\Biggl[\prod_{\stackrel{1\leq l\leq n^{(i)}_l(b(t-\sqrt{t})) }{x_i\in \GG_{bt,A,\frac{1}{2}};x_l^{(i)}\in \TT_{r,b(t-\sqrt{t})}}}
 %\eea 
This is equal to
\be\Eq(M.2.1.1)
\E\Biggl[\prod_{\stackrel{1\leq i\leq n(bt) }{x_i\in\GG_{bt,A,1/2}}}
 \Bigl\{1-C(a)((1-b)t)^{-1/2}
e^{(1-b)t-\sfrac{(\tilde m(t)+y-x_i(b t))^2}{2(1-b)t \s_2^2} }
(1+o(1))\Bigr\}\Biggr].
\ee 
Using that $x_i(bt)-\sqrt{2}\s_1^2tb\in[-A \sqrt{t}, A \sqrt{t}]$ we have the uniform bounds
\be\Eq(M.2.1.3)
\exp\left((1-b)t-\sfrac{(\tilde m(t)+y-x_i(bt))^2}{2(1-b)t \s_2^2} \right) 
\leq \exp\left((1-\s_2^2)bt+\log t+A \sqrt t\right).
\ee
Observe that the right-hand side of Eq. \eqv(M.2.1.3)$\to 0$ as $t\uparrow\infty$, since $\s_2^2>1$. 
Hence \eqv(M.2.1.1) is equal to % (up to a term of order $\e$)
\be\Eq(M.2.1.1.1)
\E\Biggl[\prod_{\stackrel{1\leq i\leq n(bt) }{x_i\in\GG_{bt,A,1/2}}}
 \exp\left(-C(a)((1-b)t)^{-1/2}
e^{(1-b)t-\sfrac{(\tilde m(t)+y-x_i(b t))^2}{2(1-b)t \s_2^2} }
(1+o(1))\right)\Biggr].
\ee 
%\bea\Eq(M.2.1.2)
%& &\E\Biggl[\prod_{\stackrel{1\leq i\leq n(b\sqrt t)}{x_i\in\GG_{b\sqrt t,B,\g} } } \E\Biggl[\exp\Big\{-\sum_{\stackrel{1\leq l\leq n^{(i)}_l(b(t-\sqrt{t})) }{x_i\in\GG_{bt,A,\frac{1}{2}};x_l^{(i)}\in \TT_{r,b(t-\sqrt{t})}}}
 %C(a)((1-b)t)^{-1/2}\\\nonumber
%&&\quad\times \exp\left((1-b)t-\sfrac{(\tilde m(t)+y-x_i(b\sqrt t)-x_l^{(i)}(b(t-\sqrt t)))^2}{2(1-b)t \s_2^2} \right)(1+o(1))\Bigr\}\big \vert\FF_{b\sqrt t}\Bigr]\Biggr].
%\e
Expanding the square in the exponent in the last line and keeping only the relevant terms yields
\be\Eq(fz.400)
\sqrt 2 y+ t\s_2^2(1-b)+2\s_1^2bt -\sqrt 2x_i(bt)+ \frac{\left(\sqrt 2t\s_1^2b-x_i(bt)\right)^2}{2(1-b)\s_2^2t}.
\ee
The terms up to the last one would nicely combine to produce the McKean martingale as coefficient of $C(a)$. However, the last terms are of order one and cannot be neglected. To deal with them,  we split the process at time $b\sqrt t$. 
We write somewhat abusively $x_i(bt)=x_i(b\sqrt t) +x_{l}^{(i)}(b(t-\sqrt t))$, where we understand that $x_i(b\sqrt t)$ is the ancestor
at time $b\sqrt t$ of 
the particle that at time $t$ is labeled $i$ if we think backwards from time $t$, while the labels of the particles at time $b\sqrt t$ run only over the different ones, i.e. up to $n(b\sqrt t)$, if we think in the forward direction. No confusion should occur if this is kept in mind.

Using Proposition \thv(Prop.p.4) and Proposition \thv(Prop.p.5) we can further localise the path of the particle.
Recall the definition of $\GG_{s,A,\g}$ and $\TT_{r,s}$, 
we rewrite \eqv(M.2.1.1.1), up to a term of order $\e$,  as 
 \bea \Eq(M.2.1.2)
&& \E\Biggl[\prod_{\stackrel{1\leq i\leq n(b\sqrt t) }{x_i\in \GG_{b\sqrt{t},B,\g}}}\E\Biggl[\prod_{\stackrel{1\leq l\leq n^{(i)}_l(b(t-\sqrt{t})) }{ x_i\in\GG_{bt,A,\frac{1}{2}};x_l^{(i)}\in \TT_{b(t-\sqrt{t}),r}}}
 \exp\Biggl(-C(a)((1-b)t)^{-1/2}\\\nonumber
&&\quad\times \exp\left((1-b)t-\sfrac{(\tilde m(t)+y-x_i(b \sqrt t)-
x_l^{(i)}(b(t-\sqrt t)))^2}{2(1-b)t \s_2^2} \right)(1+o(1))\Biggr)
\big \vert \FF_{\sqrt tb}\Biggr]\Biggr].
\eea
Using that  $x_i(b\sqrt t)+x_l^{(i)}(b(t-\sqrt{t}))-\sqrt{2}\s_1^2tb\in[-A \sqrt{t}, A \sqrt{t}]$ and $\tilde m=\sqrt{2}-\frac{1}{2\sqrt 2}\log t$, 
we can re-write the terms multiplying $C(a)$ in \eqv(M.2.1.2) as
\bea\Eq(M.2.1.4)\nonumber
%&& \exp\big((1-b)t-\frac{(\tilde m(t)+y-x_i(b\sqrt t)-x_l^{(i)}(b(t-\sqrt t)^2}{2(1-b)t \s_2^2} \big)((1-b)t)^{1/2}\nonumber\\
%&=& \exp\big((1-b)t-\frac{((\tilde m(t)-\sqrt 2\s_1^2bt)+y-(x_i(b\sqrt t)+x_l^{(i)}(b(t-\sqrt t))-\sqrt{2}\s_1^2bt))^2}{2(1-b)t \s_2^2} \big)((1-b)t)^{1/2}\nonumber\\
&& \exp\Big(-(1+\s_1^2)bt+\sqrt 2 (x_i(b\sqrt t)+x_l^{(i)}(b(t-\sqrt t)))
-\frac{1}{2}\log(1-b)-\sqrt{2}y
\\\nonumber && -\sfrac{(x_i(b\sqrt t)+x_l^{(i)}(b(t-\sqrt t))-\sqrt 2 \s_1^2bt)^2}{2(1-b)\s_2^2t}+O(1/\sqrt{t})\Big)\nonumber\\
&&\equiv E(x_i,x_l^{(i)})=E(x_i(b\sqrt t),x_l^{(i)}(b(t-\sqrt t)))=E(x_i(b\sqrt t), x_i(bt)-x_i(b\sqrt t)).\nonumber\\
\eea
Now  \eqv(M.2.1.2) takes the form
\be\Eq(M.2.1.5)
 \E\Biggl[\prod_{\stackrel{1\leq i\leq n(b\sqrt t)}{x_i\in\GG_{b\sqrt t,B,\g}} } \E\Biggl[\exp\Big\{-\sum_{\stackrel{1\leq l\leq n^{(i)}_l(b(t-\sqrt{t})) }{x_i\in\GG_{bt,A,\frac{1}{2}};x_l^{(i)}\in \TT_{r,b(t-\sqrt{t})}}}
 C(a) 
  E(x_i,x_l^{(i)})(1+o(1))\Bigr\}\big \vert\FF_{b\sqrt t}\Bigr]\Biggr].
\ee
Using the  inequalities
\be\Eq(M.2.1.6)
1-x\leq e^{-x}\leq 1-x+\frac{1}{2}x^2,\quad x>0,
\ee
 for 
 \be \Eq(M.2.1.7)
 x= \sum_{\stackrel{1\leq l\leq n^{(i)}_l(b(t-\sqrt{t})) }{x_i\in \GG_{bt,A,\frac12};x_l^{(i)}\in \TT_{r,b(t-\sqrt{t})}}}
 C(a)  
   E(x_i,x_l^{(i)})(1+o(1))
 \ee
we are able to bound \eqv(M.2.1.5) from below and above.  First, 
  \be\Eq(M.2.1.8)
  \E[x^2\vert \FF_{b\sqrt t}]\leq
  e^{-2(1+\s_1^2)b\sqrt t+2\sqrt{2}x_i(b\sqrt{t})-2\sqrt 2 y}\E\left[\left(Y_{b(t-\sqrt{t})}^{A}\right)^2\right],
  \ee
  where $Y_{b(t-\sqrt{t})}^{A}$ is the truncated McKean martingale defined in \eqv(D.6). Note that its second moment is bounded by $D_2(r)$ (see  \eqv(ui.2)).
  Second,  computing  the conditional expectation given $\FF_{b\sqrt t}$ yields  
  \bea\Eq(M.2.1.8)
\E[x\vert \FF_{b\sqrt t}]&=&  \E\Big[\sum_{\stackrel{1\leq l\leq n^{(i)}_l(b(t-\sqrt{t})) }{x_i\in \GG_{bt,A,\frac12};x_l^{(i)}\in \TT_{r,b(t-\sqrt{t}}}}
 C(a)  E(x_i,x_l^{(i)})(1+o(1))\big \vert\FF_{b\sqrt t}\Big]\\\nonumber
 &\leq& e^{b(\s_1^2t-\sqrt{t})-\sqrt 2 y} \int_{K_t-A\sqrt{t}}^{K_t+{A\sqrt{t}}}e^{\sqrt{2}(z+x_i(b\sqrt t))-\frac{(z+x_i(b\sqrt t)-\sqrt 2\s_1^2bt)^2}{2\s_2^2(1-b)t}}\sfrac{ e^{-z^2/2\s_1^2b(t-\sqrt t)}   dz}{\sqrt{2\pi \s_1^2b(t-\sqrt{t})}},
  \eea                                                                                        
 where $K_t=\sqrt{2}tb\s_1^2-x_i(b\sqrt{t})$. Performing the change of variables $z=w+K_t$ this is equal to
 \bea\Eq(M.2.1.9)
 &&e^{-(1+\s_1^2)b\sqrt{t}+\sqrt{2}x_i(bt)-\frac{1}{2}\log (1-b)-\sqrt 2 y}\int_{-A\sqrt t}^{A\sqrt t}e^{-\frac{w^2}{2\s_1^2b(t-\sqrt{t})}-\frac{w^2}{2\s_2^2(1-b)t}}\sfrac{dw}{\sqrt{2\pi\s_1^2b(t-\sqrt t)}}(1+o(1))\nonumber\\
 &=&e^{-(1+\s_1^2)b\sqrt{t}+\sqrt{2}x_i(bt)-\frac{1}{2}\log (1-b)-\sqrt 2 y} \left(\sfrac{\s_2^2(1-b)}{ 1-\s_1^2b/\sqrt t }\right)^{1/2}\int_{-A\sqrt{t}}^{A\sqrt t}e^{-w^2/2t}\sfrac{dw}{\sqrt{2\pi t}}(1+o(1))\nonumber\\
 &=&e^{-(1+\s_1^2)b\sqrt{t}+\sqrt{2}x_i(bt)-\sqrt 2 y} \left(\sfrac{\s_2^2  }{ 1-\s_1^2b/\sqrt t }\right)^{1/2}(1-\e)(1+o(1)),
 \eea
 where $o(1)\leq O(t^{\g-1})$.
 Using Lemma \thv(simplefact) together with the independence of the Brownian bridge from its endpoint, we obtain that the right hand side of  \eqv(M.2.1.9) multiplied by an  additional factor $(1-\e)$ is also a lower bound. Comparing this to \eqv(M.2.1.8), one sees that 
  \be\Eq(M.2.1.8)
 \frac{\E[x^2\vert \FF_{b\sqrt t}]}{\E[x\vert \FF_{b\sqrt t}]}\leq
   D_2(r) e^{-(1+\s_1^2)b\sqrt t+\sqrt{2}x_i(b\sqrt{t})} \leq C e^{-(1-\s_1^2)b\sqrt t+0(t^{\g/2})},
  \ee
  which tends to zero uniformly as $t\uparrow \infty$.
 Thus the second moment term is negligible. 
 Hence we only have to control
 \bea\Eq(M.2.1.10)
&& \E\Biggl[\prod_{\stackrel{1\leq i\leq n(b\sqrt t)}{x_i\in\GG_{b\sqrt t,B,\g}} }\Big(1-C(a)e^{-(1+\s_1^2)b\sqrt{t}+\sqrt{2}x_i(bt)-\sqrt 2 y} \left(\sfrac{\s_2^2  }{ 1-\s_1^2b/\sqrt t }\right)^{1/2}\Big)\Biggr]\nonumber\\
 &=&\E\Biggl[\exp\Biggl(-\sum_{\stackrel{1\leq i\leq n(b\sqrt t)}{x_i\in\GG_{b\sqrt t,B,\g}} }C(a)e^{-(1+\s_1^2)b\sqrt{t}+\sqrt{2}x_i(bt)-\sqrt 2 y} \left(\sfrac{\s_2^2  }{ 1-\s_1^2b/\sqrt t }\right)^{1/2} \Biggr)(1+o(1))\Biggr]\nonumber\\
 &=& \E\Biggl[\exp\left(-C(a)\left(\sfrac{\s_2^2  }{ 1-\s_1^2b/\sqrt t }\right)^{1/2}e^{-\sqrt 2 y} \tilde Y_{b\sqrt t,\g}^B\right)(1+o(1))\Biggr]
 \eea
where 
\be
\Eq(fz.40)
\tilde Y^{B}_{b\sqrt t,\g}=\sum_{i=1}^{n(b\sqrt t)}e^{-(1+\s_1^2)b\sqrt t+\sqrt{2}x_i(b\sqrt t)}\1_{x_i(b\sqrt{t})-\sqrt 2\s_1^2b\sqrt t\in[-Bt^{\g/2},Bt^{\g/2}]}.
\ee 
Now from Lemma \thv(tilde.1),  $\tilde Y^{B }_{b\sqrt t,\g}$ converges in probability and in $L^1$ to the random variable $Y$, when we let first $t$ and then $B$ tend to  infinity. Since $Y^{B}_{b\sqrt t,\g}\geq 0$ and $C(a)>0$, it follows 
\bea\Eq(final.1)
&&\lim_{B\uparrow\infty}\liminf_{t\uparrow\infty} \E\left[\exp\left(-C(a)\left(\sfrac{\s_2^2}{1-\s_1^2b/\sqrt t}\right)^{1/2}\tilde Y^{B}_{b\sqrt t,\g} e^{-\sqrt{2}y}\right)\right]\nonumber\\
&=&
\lim_{B\uparrow\infty}\limsup_{t\uparrow\infty} \E\left[\exp\left(-\s_2C(a)\tilde Y^{B}_{b\sqrt t,\g} e^{-\sqrt{2}y}\right)\right]\nonumber\\
&=&
\E\left[\exp\left(-\s_2C(a)Ye^{-\sqrt{2}y}\right)\right].
\eea
Finally, letting $r$ tend to $+\infty$, all the $\e$-errors (that are still present implicitly, vanish. 
This concludes the proof of Theorem \thv(M.Th1).
\end{proof}

\section{Existence of the limiting process}

The following existence theorem is the basic step in the proof of Theorem \thv(speed.1).

\begin{theorem}\TH(E.Th1) Let $\s_1<\s_2$. Then, 
 the point processes $\EE_t=\sum_{k\leq n(t)}\d_{x_k(t)-\tilde m(t)}$ converges in law to a 
 non-trivial point process $\EE$.
\end{theorem}
\begin{proof}
 It suffices to show that, for $\phi\in\CC_c(\R)$ positive, the Laplace functional
\be
\Psi_t(\phi)=\E\left[\exp\left(-\int\phi(y)\EE_t({d}y)\right)\right],
\ee
of the processes $\EE_t$  converges. First observe that this limit cannot be zero, since the maximum of the time inhomogeneous BBM converges by Theorem \thv(M.Th1). As for standard BBM (see e.g. \cite{ABK_E}), it follows
\be
\lim_{N\to\infty}\lim_{t\to\infty}\P\left[\EE_t(B)>N\right]=0,\;\mbox{for any bounded}\, B\subset \R,
\ee
which implies the locally finiteness of the limiting point process. As in 
\cite{ABK_E} we decompose  
\be
\Psi_t(\phi)= \Psi_t^{<\d}(\phi)+\Psi_t^{>\d}(\phi),
\ee
where
\bea
\Psi_t^{<\d}(\phi) &=& \E\left[\exp\left(-\int\phi(y)\EE_t({d}y)\right)\1_{\max \EE_t\leq\d}\right]\nonumber\\
\Psi_t^{>\d}(\phi) &=& \E\left[\exp\left(-\int\phi(y)\EE_t({d}y)\right)\1_{\max \EE_t>\d}\right].
\eea
Here we write shorthand $\max \EE_t\leq \d$ for 
$\max_{k\leq n(t)} (x_k(t)-m(t)\leq \d$.
By Theorem \thv(M.Th1) we have
\be
\limsup_{\d\to\infty}\limsup_{t\to\infty}\Psi_t^{>\d}(\phi)
\leq \limsup_{\d\to\infty}\limsup_{t\to\infty}\P[\max \EE_t>\d]=0.
\ee
Hence it remains to analyse the behaviour of $\Psi_t^{<\d}(\phi)$. We claim that
\be
\lim_{\d\to\infty}\lim_{t\to\infty}\Psi_t^{<\d}(\phi)=\Psi(\phi)
\ee
exists and is strictly smaller than $1$. To see this set
\be
\bar{\phi}(z)=\phi(\s_2 z)
\ee
and
\be
g_\d(z)= e^{-\bar{\phi}(-z)}\1_{\{-z\s_2\leq\d\}}.
\ee
Moreover, define
\be\Eq(E.14)
u_\d(t,z)=1-\E\left[\prod_{j\leq n(t)}g_\d(z-\bar{x}_j(t))\right] .
\ee
where $\{\bar{x}_j(t),1\leq j\leq n(t)\}$ are the particles of a standard BBM with variance $1$. We observe that by \cite{McKean} $u_\d(t,x)$ solves the
 F-KPP equation \eqv(fkpp) with initial condition $u_\d(0,x)=1-g_\d(x)$. Next we verify Assumptions (i)-(iv) of Proposition \thv(A.Prop1). (i) is clear. Moreover, $ g_\d(x)=1$ for $x$ large enough in the positive , and $ g_\d(x)=0$ for $-x$ large enough, so that Conditions (ii)-(iv) of Proposition \thv(A.Prop1) are satisfied.
 Now
\bea
\Psi_t^{<\d}(\phi)
&=&\E\left[\prod_{i\leq n(bt)}\E\left[\prod_{x_j^i\leq n_i((1-b)t)} g_\d((\tilde m(t)-x_i(bt)-x_{j}^i((1-b)t))/\s_2)\big | \FF_{bt}\right]\right]\nonumber\\
&=&\E\left[\prod_{i\leq n(bt)}\E\left[\prod_{\bar{x}_j^i\leq n_i((1-b)t)} g_\d((\tilde m(t)-x_i(bt))/\s_2-\bar{x}_{j}^i((1-b)t))\big | \FF_{bt}\right]\right],\nonumber\\
\eea
where for each $i$,  $\bar{x}_{j}^i$ are the particles of iid standard BBMs. By Proposition \thv(A.Prop1) and the same calculations as in the proof Theorem \thv(M.Th1) we have that this converges, as $t\to\infty$, to
\be\Eq(E.8)
\E\left[\exp\left(-\s_2C(a,\phi,\delta)Y\right)\right],
\ee
where $C(a,\phi,\delta)$ is the constant that appears in Lemma \thv(fz.6), with initial condition $g_\d(z)$, i.e.
\be
C(a,\phi,\delta)=\lim_{t\to\infty} \frac1{\sqrt {2\pi}}\int_0^\infty 
u_\d(t, z +\sqrt{2}t)e^{(\sqrt{2}+a) z-a^2t/2}(1-e^{-2 za}){d}z,
\ee
where $ a=\sqrt{2}(\s_2-1)$ and
 $u_\d$ is the solution to the F-KPP equation \eqv(fkpp) with initial condition $u_\d(0,z)=1-e^{-\bar{\phi}(z)}\1_{\{z\s_2\leq \d\}}$. Thus the limit $ \lim_{t\to\infty}\Psi_t^{<\d}(\phi)=\Psi^{<\d}(\phi)$ exists. 
 The limit $\d\uparrow\infty$ then exists by the same argument as in the proof
 of Theorem 3.1 of \cite{ABK_E}:
 the function
\be
\d\to\Psi^{<\d}(\phi)
\ee
is increasing and bounded, Moreover,  the maximum is an atom of $\EE_t$ and $\phi$  is nonnegative, and so 
\be\Eq(E.1)
\Psi_t^{<\d}(\phi)\leq \E\left[\exp\left(-\phi(\max\EE_t)\right)\1_{\{\max \EE_t\leq \d\}}\right]
\ee
The limit as $t\to\infty$ and $\d\to\infty$ of the right hand side of \eqv(E.1) exists by Theorem \thv(M.Th1).  Hence 
\be\Eq(E.11)
\Psi(\phi)= \lim_{\d\to\infty}\Psi^\d(\phi)<1,
\ee
by monotone convergence.
This implies the existence of the limiting process.
\end{proof}

\begin{proposition}\TH(E.Prop1)
 Let $v(t,x)$, $v_\d(t,x)$ be solutions of the F-KPP equation with initial data $v(0,x)=1-e^{-\bar\phi(-x)}$ and $v_\d(0,x)=1-e^{-\bar\phi(-x)}\1_{\{-x\s_2\leq \d\}}$ respectively. Set
\be
C(a,\phi.\d)=\lim_{t\to\infty}\frac{1}{\sqrt{2\pi}}\int_0^\infty v_\d(t,z+\sqrt{2}t)e^{(\sqrt{2}+a)z-a^2t/2}\left(1-e^{-2az}\right){d}z
\ee
Then $\lim_{\d\to\infty}C(a,\phi,\d)$ exists and is given by
\be\Eq(E.9)
C(a,\phi)=\lim_{\d\to\infty}C(a,\phi,\d)=\lim_{t\to\infty}\frac{1}{\sqrt{2\pi}}\int_0^\infty v(t,z+\sqrt{2}t)e^{(\sqrt{2}+a)z-a^2t/2}{d}z.
\ee
Moreover,
\be\Eq(E.7)
\lim_{t\to\infty}\Psi_t(\phi)=\E\left[\exp\left(-\s_2C(a,\phi)Y\right)\right].
\ee
\end{proposition}
\begin{proof} 
First we note that 
\be\Eq(fz.10)
C(a,\phi.\d)=\lim_{t\to\infty}\frac{1}{\sqrt{2\pi}}\int_0^\infty v_\d(t,z+\sqrt{2}t)e^{(\sqrt{2}+a)z-a^2t/2}{d}z.
\ee
To see this, note that for any $K<\infty$,
\be\Eq(fz.11)
\lim_{t\to\infty}\frac{1}{\sqrt{2\pi}}\int_0^K v_\d(t,z+\sqrt{2}t)e^{(\sqrt{2}+a)z-a^2t/2}{d}z
\leq \lim_{t\to\infty}\frac{1}{\sqrt{2\pi}} Ke^{-a^2t^2/2} e^{(\sqrt{2}+a)K}=0.
\ee
Hence, for any $K<\infty$, 
\bea\Eq(fz.12)
&0&\geq C(a,\phi.\d)-\limsup_{t\to\infty}\frac{1}{\sqrt{2\pi}}\int_0^\infty v_\d(t,z+\sqrt{2}t)e^{(\sqrt{2}+a)z-a^2t/2}{d}z
\\\nonumber
&&\geq C(a,\phi.\d)-\liminf_{t\to\infty}\frac{1}{\sqrt{2\pi}}\int_0^\infty v_\d(t,z+\sqrt{2}t)e^{(\sqrt{2}+a)z-a^2t/2}{d}z
\\\nonumber
&&\geq -
e^{-aK} \limsup_{t\to\infty}\frac{1}{\sqrt{2\pi}}\int_0^\infty v_\d(t,z+\sqrt{2}t)e^{(\sqrt{2}+a)z-a^2t/2}{d}z.
\eea
Since this holds for all $K$, Eq. \eqv(fz.10) follows. 
It remains to control the limit as $\d\uparrow\infty$ of the right-hand side of \eqv(fz.10). But an exact rerun of the proof of 
%
% Denote by $\{\bar x_k(t), k\leq n(t)\}$ the particles of a standard BBM. Since $\phi$ is nonnegative we have 
%\be\Eq(E.3)
%0\leq v_\d(t,x)-v(t,x)\leq\P\left[\max_{k\leq n(t)}\bar x_k(t)>\d+x\right].
%\ee
% this bound to replace $v$ in \eqv(E.9) by $v_\d$, 
%This implies 
%\bea\Eq(E.2)
%C(a,\phi,\d)
%&\leq &
%\lim_{t\to\infty}\frac{1}{\sqrt{2\pi}}\int_0^\infty \P\left[max_{k\leq n(t)}x_k(t)-\sqrt{2}>\d+x]e^{(\sqrt{2}+a)z}e^{-a^2t/2}\left(1-e^{-2az}\right){d}z\nonumber\\
%&\leq&\lim_{t\to\infty}\frac{1}{\sqrt{2\pi}}\int_0^\infty v_\d(t,z+\sqrt{2}t) e^{(\sqrt{2}+a)z-a^2t/2} {d}z\\\nonumber
%&+& \lim_{t\to\infty}\frac{1}{\sqrt{2\pi}}\int_0^\infty \P\left[\max_{k\leq n(t)}\bar x_k(t)-\sqrt{2}>\d+z\right]e^{(\sqrt{2}-a)z-a^2t/2} {d}z
%\eea
%The bound of Lemma \thv(E.Lem1) implies that the second summand in \eqv(E.2) is bounded by
%\be
%\lim_{t\to\infty}\int_0^\infty e^{-(z+at)^2/2t+o(t)}dz = 0.
%\ee
%Hence 
%\be\Eq(E.6)
%C(a,\phi,\delta)
%=\lim_{t\to\infty}\frac{1}{\sqrt{2\pi}}\int_0^\infty v_\d(t,z+\sqrt{2}t)e^{(\sqrt{2}+a)z-a^2t/2} {d}z.
%\ee
%An exact rerun of the proof
 Lemma 4.10 in \cite{ABK_E} using Lemma \thv(E.Lem2) below instead of 
Lemma 4.8 of \cite{ABK_E} yields that 
\be\Eq(E.12')
\lim_{\d\uparrow\infty}\lim_{t\uparrow\infty}\int_0^\infty v_\d(t,x+\sqrt{2}t)e^{(\sqrt{2}+a)z-a^2t/2}{d}z\equiv\lim_{\d\uparrow\infty}  F(\d)\equiv F
\ee
exists. 
 By \eqv(E.8) and \eqv(E.12') we have
\be
\lim_{t\to\infty}\Psi_t^{<\d}(\phi)=\E\left[\exp\left(-\s_2C(a,\phi,\delta)Y\right)\right]=\E\left[\exp\left(-\frac{\s_2}{\sqrt{2\pi}}F(\d)Y \right)\right].
\ee
 This converges for $\d\to\infty$ to 
\be\Eq(E.10)
\E\left[\exp\left(-\frac{\s_2}{\sqrt{2\pi}}FY\right)\right].
\ee
Hence $F=0$ would imply
\be
\lim_{\d\to\infty}\lim_{t\to\infty}\Psi_t(\phi)=1,
\ee
which contradicts \eqv(E.11) and Theorem \thv(E.Th1). Hence $F>0$. Moreover, \eqv(E.10) implies \eqv(E.7), which concludes the proof of Proposition \thv(E.Prop1). 
\end{proof}

We recall the following estimate for the tail probabilities of standard BBM.

\begin{lemma}[\cite{ABK_P}, Corollary 10]\TH(E.Lem1) There exists $t_0<\infty$, such that 
 for $z>1$ and $t\geq t_0 $ 
\be
\P\left[\max_{k\leq n(t)}\bar x_k(t)-\sqrt{2}t+\frac{3}{2\sqrt{2}}\log t\geq z\right]\leq \rho z\exp\left(-\sqrt{2}z-\frac{z^2}{2t}+\frac{3z}{2\sqrt{2}}\frac{\log t}{t}\right),
\ee
for some constant $\rho>0$.
\end{lemma}
\begin{lemma}\TH(E.Lem2)
 Let $u$ be a solution of the F-KPP equation with initial data satisfying Assumptions (i)-(iv) of Proposition \thv(A.Prop1). Let 
\be\Eq(fz.13)
C(a)=\lim_{t\to\infty}\frac{1}{\sqrt{2\pi}}\int_0^\infty u(t,z+\sqrt{2}t)e^{(\sqrt{2}+a)z-a^2t/2}{d}z.
\ee
Then for any $x\in\R$:
\be\Eq(E.12)
\lim_{t\to\infty}\frac{1}{\sqrt{2\pi}}\int_0^\infty u(t,x+z+\sqrt{2}t)e^{(\sqrt{2}+a)z-a^2t/2}{d}z=C(a)e^{-(\sqrt{2}+a)x}.
\ee
Moreover, for any bounded continuous function $h(x)$, that is zero for $x$ small enough
\bea\Eq(E.13)
& &\lim_{t\to\infty}\int_{-\infty}^0 \E\left[h\left(y+\max \bar{x}_i(t)-\sqrt{2}t\right)\right]\frac{1}{\sqrt{2\pi}}e^{-(\sqrt{2}+a)y-a^2t/2}{d}y\nonumber\\
&=& C(a)\int_\R h(z)(\sqrt{2}+a)e^{-(\sqrt{2}+a)z}{d}z,
\eea
where $\{\bar{x}_i(t),i\leq n(t)\}$ are the particles of a standard BBM with variance $1$. Here $C(a)$ is the constant from \eqv(fz.13) for $u$ satisfying the initial condition $\1_{\{x\leq 0\}}$. 
\end{lemma}

\begin{proof}
 We have by a simple change of variables
\bea
 &&\frac{1}{\sqrt{2\pi}}\int_0^\infty u(t,z+\sqrt{2}t)e^{(\sqrt{2}+a)z-a^2t/2}{d}z
\\\nonumber
&=&\frac{e^{(\sqrt 2+a)x}}{\sqrt{2\pi}}\int_{-x}^\infty u(t,x+z+\sqrt{2}t)e^{(\sqrt{2}+a)z-a^2t/2}{d}z.
\eea
Moreover, $\lim_{t\to\infty} u(t,x+z+\sqrt{2}t)=0$ implies
\be
\lim_{t\to\infty}\frac{1}{\sqrt{2\pi}}\int_{-x}^0 u(t,x+z+\sqrt{2}t)e^{(\sqrt{2}+a)z-a^2t/2}{d}z=0,
\ee
which  proves \eqv(E.12). Moreover, \eqv(E.12) with initial condition $\1_{\{x\leq0\}}$ implies that \eqv(E.13) holds for $h(x)=\1_{[b,\infty)},b\in\R$. For general $h$ \eqv(E.13) follows in the same way as Lemma 4.11 in \cite{ABK_E} by linearity and a monotone class argument.
\end{proof}

\section{The auxiliary process}

We define the following auxiliary process that  has the same limiting behaviour as that of the two-speed BBM.
We will denote the law of these processes by $P$ and expectations by $E$. If desired, 
all ingredients of the auxiliary process can be thought of to be defined on a new probability space.
 Let $(\eta_i;i\in\N)$ be the atoms of a Poisson point process $\eta$ on $(-\infty,0)$ with intensity measure
\be
\sfrac{\s_2}{\sqrt{2\pi}}e^{-(\sqrt{2}+a)z}   e^{-a^2t/2}
{d}z.
\ee
For each $i\in\N$ consider independent standard BBMs $\bar x^{i}$. 
The auxiliary point process of interest is the superposition of the i.i.d BBMs with drift  shifted by $\eta_i+\frac{1}{\sqrt{2}+a}\log Y$, where $a$ is the constant defined in \eqv(constant.a):
\be
\Pi_t=\sum_{i,k}\d_{\left(\eta_i+\frac{1}{\sqrt{2}+a}\log Y+\bar x_k^i(t)-\sqrt{2}t\right)\s_2}.
\ee
\begin{remark} 
The form of  the auxiliary process is similar to the case of standard BBM, but with a different intensity of the Poisson process. In particular, the intensity decays exponentially with  $t$. This is a consequence of the fact that particles at the time of the speed change were forced to be $O(t)$ below the line $\sqrt 2 t$, in contrast to the $O(\sqrt t)$ in the case of ordinary BBM. The reduction of the intensity  of the process with $t$ 
forces the particles to be selected at these locations.
\end{remark}

\begin{theorem}\TH(Au.Th1)
Let $\EE_t$ be the extremal process of the two-speed BBM. Then
\be
\lim_{t\to\infty}\EE_t \stackrel{law}{=} \lim_{t\to\infty}\Pi_t.
\ee
\end{theorem}
\begin{proof}
 Using the notation $\bar{\phi}(z)=\phi(\s_2 z)$ and by the form of the Laplace functional of a Poisson point process we have
\bea\Eq(Au.8)
& & E\left[\exp\left(-\int\phi(x)\Pi_t({d}x)\right)\right]\nonumber\\
&=& E\Biggl[\exp\Biggl(-\s_2\int_{-\infty}^0\Biggl\{1-E\left[\exp\left(-\sum_{k\leq n(t)} \bar{\phi}\left(z+\bar x_k(t)-\sqrt{2}t+\frac{\log Y}{\sqrt{2}+a}\right)\right)\right]\Biggr\}\nonumber\\
&&\qquad  \times e^{-(\sqrt{2}+a)z}e^{-a^2t/2}{d}z\Biggr)\Biggr]\nonumber\\
&=& E\left[\exp\left(\frac{\s_2}{\sqrt{2\pi}}\int_0^\infty u\left(t,z+\sqrt{2}t-\frac{1}{\sqrt{2}+a}\log Y\right)e^{(\sqrt{2}+a)z}e^{-a^2t/2}{d}z\right)\right],
\eea
with
\be
u(t,x)=1-E\left[\exp\left(-\sum_{k\leq n(t)}\bar{\phi}(-x+\bar x_k(t)\right)\right].
\ee
By Lemma \thv(E.Lem2) we have
\bea
&&\lim_{t\to\infty}\frac{1}{\sqrt{2\pi}}\int_0^\infty u\left(t,z+\sqrt{2}t-\frac{1}{\sqrt{2}+a}\log Y\right)e^{(\sqrt{2}+a)z}e^{-a^2t/2}{d}z\nonumber\\
&=&Y\lim_{t\to\infty}\frac{1}{\sqrt{2\pi}}\int_0^\infty u(t,z+\sqrt{2}t)e^{(\sqrt{2}+a)z}e^{-a^2t/2}{d}z,
\eea
which exists and is strictly positive by Proposition \thv(E.Prop1). This implies that the Laplace functionals of $\lim_{t\to\infty}\Pi_t$ and of the extremal process of the two-speed  BBM are equal.
\end{proof}

The following proposition shows that in spite of the different Poisson ingredients, when we look at the process of the extremes of each of the $x^i(t)$, we end up with a Poisson point process just like in the standard BBM case. 

\begin{proposition}\TH(Au.Prop1)
Define the point process
\be\Eq(fz.14)
\Pi^{ext}_t\equiv\sum_{i,k}\d_{\left(\eta_i+\frac{1}{\sqrt{2}+a}\log Y+\max_{k\leq n_i(t)}\bar x_k^i(t)-\sqrt{2} t\right)\s_2}.
\ee
Then 
\be\Eq(fz.15)
\lim_{t\to\infty}\Pi^{ext}_t \stackrel{law}{=} P_Y\equiv \sum_{i\in \N}\d_{p_i},  %PPP\left(C(a)Y\sqrt{2}e^{-\sqrt{2}x}{d}x\right),
\ee
where $P_Y$ is the Poisson point process on $\R$ with intensity measure %PPP
$\s_2C(a)Y\sqrt{2}e^{-\sqrt{2}x}{d}x$.
\end{proposition}

\begin{proof}
 We consider the Laplace functional of $\Pi^{ext}_t$. Let $M^{(i)}(t)=\max \bar x_k^{(i)}(t)$ and as before $\bar{\phi}(z)=\phi(\s_2 z)$. We want to show
\bea
&&\lim_{t\uparrow\infty}E\left[\exp\left(-\sum_{i}\bar{\phi}(\eta_i+M^{(i)}(t)-\sqrt{2}t \right)\right]\nonumber\\
&=& \exp\left(-\s_2C(a)\int_{-\infty}^{\infty} \left(1-e^{-\phi(x)}\right)\sqrt{2}e^{-\sqrt{2}x}{d}x\right).
\eea
Since $\eta_i$ is a Poisson point process and the $M^{(i)}$ are i.i.d. we have
\bea
&&E\left[\exp\left(-\sum_{i}\bar{\phi}(\eta_i+M^{(i)}(t)-\sqrt{2}  t\right)\right]\nonumber\\
&=&\exp\left(-\s_2\int_{-\infty}^0E\left[1-e^{-\bar{\phi}(z+M(t)-\sqrt{2} t)}\right]
e^{-(\sqrt{2}+a)z-a^2t/2 }\frac{{d}z}{\sqrt{2\pi}}\right),
\eea
where $M(t)$ has the same distribution as one the variables $M^{(i)}(t)$. Now we apply Lemma \thv(E.Lem2) with $h(x)=1-e^{-\bar{\phi}(z)}$. Hence the result follows  by using that $\bar{\phi}(z)=\phi(\s_2 z)$ and $\sqrt{2}+a=\sqrt{2}\s_2$ together with the change of variables $x=\s_2z$.
\end{proof}

The following proposition states that the Poisson points of the auxiliary process contribute to the limiting process come from a neighbourhood of 
$-at$. 

\begin{proposition}\TH(Au.Prop2)
Let $z\in\R,\e>0$.  Let $\eta_i$ be the atoms of a Poisson point process with intensity 
measure $Ce^{-(\sqrt 2+a)x-a^2t/2}dx$ on  $(-\infty,0]$. Then there exists $B<\infty$ such that 
\be
\sup_{t\geq t_0}P\left(\exists i,k:\eta_i+\bar x_k^i(t)-\sqrt{2}t\geq z, \eta_i\not\in[-at-B\sqrt{t},-at+B\sqrt{t}]\right)\leq \e.
\ee
\end{proposition}
\begin{proof}   By a first order Chebychev inequality we have
\bea\Eq(Au.3)
&&P\left(\exists i,k:\eta_i+\bar x_k^{(i)}(t)-\sqrt{2}t\geq z, \eta_i>-at+B\sqrt{t}\right)\nonumber\\
&\leq& C\int_{-at+B\sqrt{t}}^0 P\left(\max \bar x_k(t)\geq \sqrt{2}t-x+z\right)e^{-(\sqrt{2}+a)x}e^{-a^2t/2}{d}x\nonumber\\
&=&C\int_0^{at-B\sqrt{t}} P\left(\max \bar x_k(t)\geq \sqrt{2}t+x+z\right)e^{(\sqrt{2}+a)x}e^{-a^2t/2}{d}x,
\eea
by the change of variables $x\to-x$. Using the asymptotics of Lemma  \thv(E.Lem1) we can bound \eqv(Au.3) from above by
\bea
& &\rho C\int_{0}^{at-B\sqrt{t}}t^{-1/2}e^{-\sqrt{2}(x+z)}e^{-(x+z)^2/2t}e^{(\sqrt{2}+a)x}e^{-a^2t/2}{d}x\nonumber\\
&\leq& \rho C\int_{-a\sqrt{t}}^{-B}e^{z^2/2}{d}z(1+o(1)),
\eea
by changing  variables  $x\to x/\sqrt{t} -a\sqrt{t}$. This is a Gaussian integral and can be made as small as we want by choosing $B$ large enough. 
Similarly  one bounds 
\be\Eq(fz.151)
P\left(\exists i,k:\eta_i+x_k^i(t)-\sqrt{2}t\geq z, \eta_i<-at-B\sqrt{t}\right)
\leq C\rho \int_B^\infty e^{z^2/2}{d}z(1+o(1)).
\ee
This concludes the proof. 
\end{proof}

The next proposition describes the law of the clusters $ \bar x_k^{(i)}$.
This is analogous to Theorem 3.4 in \cite{ABK_E}.
 
\begin{proposition}\TH(Au.Prop3)
 Let $x=at+o(t)$ and $\{\tilde x_k(t),k\leq n(t)\}$ be a standard BBM
 under the conditional law   $P\left(\cdot \big |\{\max  \tilde x_k(t)-\sqrt{2}t-x>0\}\right)$. Then 
the point process
\be
\sum_{k\leq n(t)} \d_{ \tilde x_k(t)-\sqrt{2}t-x}
\ee
converges in law under  $P\left(\cdot \big |\{\max  \tilde x_k(t)-\sqrt{2}t-x>0\}\right)$
as $t\to\infty$ to a well defined point process $\bar{\EE}$. The limit does not depend on $x-at$ and the maximum of $\bar{\EE}$ shifted by $x$ has the law of an exponential  random variable with parameter $\sqrt 2+a$.
\end{proposition}

\begin{proof}
 Set $\bar{\EE}_t=\sum_k \d_{\tilde x_k(t)-\sqrt{2}t}$ and $\max \bar{\EE}_t=\max \tilde x_k(t)-\sqrt{2}t$. First we show that for $X>0$
\be\Eq(Au.4)
\lim_{t\to\infty}P\left(\max \bar{\EE}_t>X+x\vert \max \bar{\EE}_t>x\right)=e^{-(\sqrt{2}+a)X}.
\ee
To see this we rewrite the conditional probability as $\frac{P\left[\max \bar{\EE}_t>X+x\right]}{P\left[\max \bar{\EE}_t>x\right]}$ and use the uniform bounds of Proposition 4.3 in \cite{ABK_E}. Observing that
\be
\lim_{t\to\infty}\frac{\Psi(r,t,X+x+\sqrt{2}t)}{\Psi(r,t,x+\sqrt{2}t)}=
e^{-(\sqrt{2}+a)X},
\ee 
where $\Psi$ is defined in Equation \eqv(A.1),
we get \eqv(Au.4) by first taking $t\to\infty$ and then $r\to\infty$. The general claim of Proposition \thv(Au.Prop3) follows in exactly the same way from \eqv(Au.4) as Theorem  3.4. in \cite{ABK_E}.
\end{proof}

Define the gap process
\be\Eq(fz.17)
D_t=\sum_k \d_{\tilde x_k(t)-\max_j \tilde x_j(t) }.
\ee
Denote by $\xi_i$ the atoms of the limiting process $\bar {\EE}$, i.e.
 $\bar{\EE}\equiv \sum_j\d_{\xi_j}$ and  define
\be\Eq(Au.7)
D\equiv\sum_j \d_{\L_j},\quad \L_j=\xi_j-\max_i \xi_i.
\ee
$D$ is a point process on $(-\infty.0]$ with an atom at $0$.

\begin{corollary} \TH(Au.Cor1)
Let $x=-at+o(t)$. In the limit $t\to\infty$ the random variables $D_t$ and $x+\max \bar \EE_t$ are conditionally independent on the event $\{x+\max \bar{\EE}_t>b\}$ for any $b\in\R$. More precisely, for any bounded function $f,h$ and $\bar{\phi}\in C_c(\R)$,
\bea\Eq(Au.5)
&&\lim_{t\to\infty}E\left[f\left(\int \bar{\phi}(z)D_t({d}z)\right)h(x+\max \bar{\EE})\big|x+\max \bar{\EE}>b\right]\nonumber\\
&=& E\left[f\left(\int \bar{\phi}(z)D({d}z)\right)\right]\frac{\int_b^\infty h(z)(\sqrt{2}+a)e^{-(\sqrt{2}+a)z}{d}z}{e^{-(\sqrt{2}+a)b}}.
\eea
\end{corollary}

\begin{proof} The proof is essentially identical to the proof of Corollary 4.12  in \cite{ABK_E}.
Let us outline, for the benefit of the readers, the structure of the proof. First, 
by Proposition \thv(Au.Prop3) the pair $(\bar\EE_t, \max\bar\EE_t-x)$, converge under the law conditioned on $\max\bar\EE_t-x>0$ 
 to $(\EE,e)$, where $e$ is an exponential random variable with parameter
$\sqrt 2+a$ and $\EE$ is independent of the precise value of the conditioning. A general continuity lemma, stated and proven  as Lemma 4.13 in \cite{ABK_E},
shows that this implies the convergence of the processes $(D_t, \max\bar\EE_t-x)$ to 
a pair $(\DD,e)$
where $D_t$ is given in \eqv(fz.17) is related to $\bar\EE_t$ by a random shift of its atoms.
The fact that $\DD$ and $e$ are independent follows from an explicit computation, just as in the proof of Corollary 4.12 in \cite{ABK_E}. We do not repeat the details.
\end{proof}

Finally we come to the description of the extremal process as seen from the Poisson process of cluster extremes, which is the formulation of Theorem 
\thv(speed.1).
%Define
%\be\Eq(Au.6)
%P_Y=\sum_{i\in\N}\d_{p_i}
%\stackrel{law}{=}
%PPP\left(C(a)Y(\sqrt{2}+a)e^{-(\sqrt{2}+a)z}{d}z\right)=
%PPP\left(C(a)Y\sqrt{2}e^{-\sqrt{2}x}{d}x\right).
%ee

\begin{theorem}\TH(Au.Th2)
Let $P_Y$ be as in \eqv(fz.15) and let $\{D^{(i)},i\in\N\}$ be a family of independent copies of the gap-process \eqv(Au.7)
with atoms $\L^{(i)}_j$. Then the point process $\EE_t$ converges in law as $t\to\infty$ to a Poisson cluster point process $\EE$ given by
\be
\EE\stackrel{law}{=}\sum_{i,j}\d_{p_i+\s_2\L_j^{(i)}}.
\ee
\end{theorem}

\begin{proof} Also this proof is now very close to that of Theorem 2.1 in \cite{ABK_E}. 
First note that the Laplace functional of the process $\EE$ is given by 
\bea\Eq(fz.20)
&&E \left[\exp\left(-\int \phi(x)\EE(dx)\right)\right]
\\\nonumber
&&=E\left[\exp\left(-C(a)Y\int_\R E\left[1-e^{-\int \phi(y+x)D( dx)}\right]\sqrt{2} e^{-\sqrt{2}y}{d}y\right)\right].
\eea
Thus, by Theorem \thv(Au.Th1), we have  to show that
the Laplace functional of the processes $\Pi_t$ converge to this 
expression. In the proof of that theorem, we have seen that 
\bea\Eq(fz.21)
&&
\lim_{t\uparrow\infty}E\left[\exp\left(-\int \phi(x)\Pi_t(dx)\right)\right]\\
&&
\nonumber
=E\left[\exp\left(-Y\lim_{t\uparrow\infty}
\int_{-\infty}^0E\left[1-\exp\left(-\int \bar{\phi}(z+x)\bar{\EE}_t({d}x)\right)\right]\frac{e^{-(\sqrt{2}+a)z-a^2t/2}}{\sqrt{2\pi}}{d}z\right)\right].
%&=&\E\left[\exp\left(-C(a)Y\int_\R E\left[1-e^{-\int \phi(y+x)D( dx)}\right]\sqrt{2} e^{-\sqrt{2}y}{d}y\right)\right]
\eea
We rewrite 
\bea
&& \int_{-\infty}^0E\left[1-\exp\left(-\int \bar{\phi}(z+x)\bar{\EE}_t({d}x)\right)\right]\frac{1}{\sqrt{2\pi}}e^{-(\sqrt{2}+a)z-a^2t/2}{d}z
\nonumber\\=&&
\Eq(Au.9)
 \int_{-\infty}^0E\left[f\left(\int\left\{T_{z+\max \bar\EE_t}\bar{\phi}(x)\right\}D_t({d}x)\right)\right]\frac{1}{\sqrt{2\pi}}e^{-(\sqrt{2}+a)z-a^2t/2}{d}z,
\eea
where $f(x)=1-e^{-x}$,  $T_z\bar{\phi}(x)=\bar{\phi}(z+x)$, $f(0)=0$. Using the localisation estimate of Proposition \thv(Au.Prop2) we have that \eqv(Au.9) is equal to
\be\Eq(fz.22)
\Omega_t(B)+\int_{-at-B\sqrt{t}}^{-at+B\sqrt{t}}E\left[f\left(\int\left\{T_{z+\max \bar\EE_t}\bar{\phi}(x)\right\}D_t({d}x)\right)\right]\frac{1}{\sqrt{2\pi}}e^{-(\sqrt{2}+a)z-a^2t/2}{d}z,
\ee
where $\lim_{B\to\infty}\sup_{t\geq t_0}\Omega_t(B)=0$. 
Let $m_{\bar\phi}$ be the minimum of the support of $\bar{\phi}$. we observe that
\be
f\left(\int\left\{T_{z+\max \bar\EE_t}\bar{\phi}(x)\right\}D_t({d}x)\right)=0,
\ee
when $z+\max \bar{\EE_t}<m_{\bar{\phi}}$. Moreover, $P\left[z+\max \bar{\EE}_t=m_{\bar{\phi}}\right]=0$. Hence
\bea\Eq(fin.1)
& &E\left[f\left(\int\left\{T_{z+\max\bar \EE_t}\bar{\phi}(x)\right\}D_t({d}x)\right)\right]\\
&=&E\left[f\left(\int\left\{T_{z+\max \bar\EE_t}\bar{\phi}(x)\right\}D_t({d}x)\right)\1_{\{z+\max \bar{\EE}_t>m_{\bar{\phi}}\}}\right]\nonumber\\
&=&E\left[f\left(\int\left\{T_{z+\max \bar\EE_t}\bar{\phi}(x)\right\}D_t({d}x)\right)\big | z+\max\bar\EE_t>m_{\bar{\phi}}\right] P\left[z+\max \bar{\EE}_t>m_{\bar{\phi}}\right].\nonumber
\eea
Now by Corollary \thv(Au.Cor1), for $z$ in the range of integration in 
\eqv(fz.22), on the event we are conditioning on in \eqv(fin.1), 
the random variables $D_t$ and $\max\bar\EE_t+z-m_{\bar\phi}$ converge to independent random variables $(D, e)$, where $e$ is exponential with 
parameter $\sqrt 2+a$.  Hence 
\bea\Eq(fz.23)\nonumber
&&\lim_{t\uparrow\infty}E\left[f\left(\int\left\{T_{z+\max \bar\EE_t}\bar{\phi}(x)\right\}D_t({d}x)\right)\big | z+\max\bar\EE_t>m_{\bar{\phi}}\right] 
%P\left[z+\max \bar{\EE}_t>m_{\bar{\phi}}\right]
\\ &&= \int_{0}^\infty  (\sqrt 2+a)e^{-(\sqrt 2+a)u}
E\left[f\left(\int\bar{\phi}(u+m_{\bar\phi}+x)D({d}x)\right)\right]du
\\\nonumber
&&
 \int_{m_{\bar\phi}}^\infty  (\sqrt 2+a)e^{-(\sqrt 2+a)(u-m_{\bar\phi})}
E\left[f\left(\int\bar{\phi}(u+x)D({d}x)\right)\right]du.
\eea
Note that this expression is independent of $z$. Thus it remains to compute the integral of
$P\left[z+\max \bar{\EE}_t>m_{\bar{\phi}}\right]$. But this converges to 
$e^{-(\sqrt 2+a)m_{\bar\phi}}$ by \eqv(E.12) in Lemma \thv(E.Lem2), 
together with the localisation estimates of Proposition \thv(Au.Prop2)
(which this time allows to re-extend the range of integration). Putting this together with \eqv(fz.23) and changing variables $x=\s_2z$ shows that
the right-hand side of \eqv(fz.21) is indeed equal to the right-hand side of 
\eqv(fz.20). This proves the theorem.
\end{proof}

\section{The case $\s_1>\s_2$}

In this section we proof Theorem \thv(speed.2). The existence of the process $\EE$ from \eqv(speed.5) will be a byproduct of the proof.

The following result  is contained in the  calculation of the maximal displacement in \cite{FZ_RW}.

\begin{lemma}\TH(K.Lem1)(\cite{FZ_RW})
For all $\e>0,d\in\R$ there exists a constant $D$ large enough such that for $t$ sufficiently large
\be
\P\left[\exists k\leq n(t):x_k(t)>m(t) +d \mbox{ \rm and } x_k(bt)<m_1(bt)-D\right]<\e.
\ee
\end{lemma}

\begin{proof}[Proof of Theorem \thv(speed.2)]  
First we  establish the existence of a limiting process. 
Note that $m(t)=m_1(bt)+m_2((1-b)t)$, where
 $m_i(s)=\sqrt{2}\s_is-\frac{3}{2\sqrt{2}}\s_i\log s$. Recall  
\be
\bar{\phi}(z)=\phi(\s_2 z)
\ee
and
\be
g_\d(z)= e^{-\bar{\phi}(-z)}\1_{\{-z\leq\d\}}.
\ee
Using that the maximal displacement is $m(t)$ in this case we can proceed as in the  proof of Theorem \thv(E.Th1) up to \eqv(E.14) and only have to control
\be\Eq(K.2)
\Psi_t^{<\d}(\phi)=\E\left[\prod_{i\leq n(tb)}\E\left[\prod_{j\leq n_i((1-b)t)} g_\d((m(t)-x_i(bt))/\s_2-\bar{x}_{j}^i((1-b)t))\big | \FF_{tb}\right]\right],
\ee
where $\bar{x}_{j}^i((1-b)t)$ are the particles of a standard BBM 
 at time $(1-b)t$ and $x_i(bt)$ are the particles of a BBM with variance $\s_1$ at time $bt$. 
 Using Lemma \thv(K.Lem1) and Theorem 1.2 of \cite{FZ_RW}
 as in the proof of Theorem \thv(M.Th1) above,  we obtain that  \eqv(K.2),  for $t$ sufficiently large, equals 
\be\Eq(K.3)
\\E\left[\prod_{\stackrel{i\leq n(bt)}{x_i(bt)>m_1(bt)-D}}\E\left[\prod_{j\leq n_i((1-
b)t)} g_\d(\sfrac{(m(t)-x_i(bt))}{\s_2}-\bar{x}_{j}^i((1-b)t))\big | \FF_{tb}\right]
\right] +O(\e).
\ee
The rest of the proof has an iterated structure. In a first step we show that 
conditioned on $\FF_{bt}$ for each $i\leq n(bt)$ the points $\{x_{i}(bt)+x_j^{i}
((1-b)t)-m(t)\vert x_i(bt)>m_1(bt)-D\}$ converge to the corresponding points of 
the 
 point process  $x_i(bt)-m_1(bt)+\s_2\tilde \EE^{(i)}$, where $\tilde \EE^{(i)}$ are independent copies of the extremal process \eqv(old.3) of standard BBM. To this end observe that 
\be 
u_\d((1-b)t,z)=1-\E\left[\prod_{j\leq n((1-b)t)} g_\d(z-\bar{x}_{j}^i((1-b)t)) 
\right]
\ee 
 solves the F-KPP  equation \eqv(fkpp) with initial condition 
 $u_{\d}(0,z)=1-e^{-\bar{\phi}(-z)}
 \1_{\{-z\leq \d\}}$. Moreover, the assumptions of  Lemma 4.9 in \cite{ABK_E} are satisfied. Hence \eqv(K.3) is equal to
\be\Eq(K.4)
 \e+ \E\left[\prod_{\stackrel{i\leq n(bt)}{x_i(bt)> m_1(bt)-D}}\left( \E\left[e^{-C(\bar\phi,\d)Ze^{-\sqrt{2}\frac{m_1(bt)-x_i(bt)}{\s_2}}}\big|\FF_{bt}\right]
 (1+o(1))\right)\right].
\ee
Here $C(\bar\phi,\d)$ is from standard BBM, i.e.
\be\Eq(normal.1)
C(\bar\phi,\d)=\lim_{t\uparrow\infty} \sqrt {\sfrac 2\pi}\int_0^\infty
u_\d(t,y+\sqrt 2 t)ye^{\sqrt 2y}dy,
\ee
see Eq. 4.49 in \cite{ABK_E}.  Note furthermore that already in \eqv(K.4) the concatenated structure of the limiting point process becomes visible.
In a second step we establish that the points $x_i(bt)-m_1(t)$ that have a descendant in the lead at time $t$ converge to $\wt\EE$.

Define 
\be\Eq(seltsam.1)
h_{\d,D}(y)\equiv 
\begin{cases}
\E\left[\exp\left(-C(\bar \phi,\d)Ze^{-\sqrt{2}\frac{\s_1}{\s_2}y}\right)\right], &\,\,\mbox{ \rm if }\,\, \s_1y<D,
\\
1, &\,\,\mbox{ \rm if }\,\, \s_1y\geq D.
\end{cases}
\ee
Then the expectation in \eqv(K.4) can be written as (we ignore the error term $o(1)$ which is easily controlled using that the probability that the number of 
terms in the product is larger than $N$ tends to zero as $N\uparrow\infty$,
uniformly in $t$)
\be\Eq(seltsam.2)
\E\left[\prod_{i\leq n(bt)} h_{\d,D}(m_1(bt)/\s_1-\bar x_i(t))\right],
\ee 
where now $\bar x$ is standard BBM. 
Defining

% For technical reason we have to introduce a second truncation.
%Let $\d_2>0$. We split the expectation in \eqv(K.4) into two parts, $\Psi^{<,\d,\d_2}(t)$, containing the $x_i(bt)-m_1<\d_2$ and $\Psi^{>,\d,\d_2}(t)$, containing the $x_i(bt)-m_1\geq\d_2$.
%$\E\left[\prod_{i\leq n(t/2)}\1_{\{x_i(t/2)> m_1(bt)-D\}}1-\E\left[\exp\left(C(\phi,\d)Ze^{-\sqrt{2}\frac{-( x_i(t/2)-m_1(bt))}{\s_2}}\right)\right]\right]$ into two bits
% \bea\Eq(K.5)
% \Psi^{<,\d,\d_2}(t)
% &=&\E\left[\prod_{i\leq n(t/2)}\1_{\{x_i(t/2)\in[ m_1(bt)-D,  m_1(bt)+\d_2]\}}\left(1-\E\left[\exp\left(C(\phi,\d)Ze^{-\sqrt{2}\frac{-(x_i(t/2)-m_1(bt))}{\s_2}}\right)\right]\right)\right]\nonumber\\
% \Psi^{>,\d,\d_2}(t)
% &=&\E\left[\prod_{i\leq n(t/2)}\1_{\{x_i(t/2)> m_1(bt)+\d_2\}}\left(1-\E\left[\exp\left(C(\phi,\d)Z^{(i)}e^{-\sqrt{2}\frac{-(x_i(t/2)-m_1(bt))}{\s_2}}\right)\right]\right)\right]
% \eea
%Observe that by the convergence of the maximal displacement of BBM
%\be
%\limsup_{\d\uparrow\infty} \limsup_{\d_2\uparrow\infty}\limsup_{t\uparrow\infty}
%\Psi^{>,\d_2}\leq\limsup_{\d_2\uparrow\infty}\limsup_{t\uparrow\infty}\P\left[\max_i x_i(bt)-m_1(bt)>\d_2\right]=0,
%\ee
%by the the size of the maximal displacement of a BBM. 
\be\Eq(seltsam.4)
v_{\d,D}(t,z)=1-\E\left[\prod_{i\leq n(t)} h_{\d,D}(z-\bar{x}_i(bt))\right],
\ee
$v_{\d,D} $ is a solution of the F-KPP equation \eqv(fkpp) with initial condition
$v_{\d,D}(0,z)= 1-h_{\d,D}(z)$. But this initial condition satisfies the 
assumptions of Bramson's Theorem A in \cite{B_C} and therefore,  
\be\Eq(seltsam.5)
v_{\d,D}(t,m(t)+x)\to \E\left[e^{-\tilde C(D,Z,C(\bar\phi,\d))\tilde Ze^{-\sqrt 2 x}}\right].
\ee
where $\tilde Z$ is an independent copy of $Z$ and 
\be\Eq(normal.2)
\tilde C(D,Z,C(\bar\phi,\d))=\lim_{t\uparrow\infty} \sqrt {\sfrac 2\pi}\int_0^\infty
v_{\d,D}(t,y+\sqrt 2 t)ye^{\sqrt 2y}dy.
\ee
 By the same argumentation as in standard BBM  setting (see \cite{ABK_E})  one obtains that 
\be
\tilde C(Z,C(\bar\phi,\d))\equiv \lim_{D\uparrow\infty}\tilde C(D,Z,C(\bar\phi,\d))= \lim_{t\uparrow\infty} \sqrt {\sfrac 2\pi}\int_0^\infty
v_\d(t,y+\sqrt 2 t)ye^{\sqrt 2y}dy,
\ee
where $v_\d$ is the solution of the F-KPP equation with  initial condition $v(0,z)=1-h_{\d}(z)$ with
\be
h_{\d}(z)=
\E\left[\exp\left(-C(\bar\phi,\d)Ze^{-\sqrt{2}\frac{\s_1}{\s_2}z}
\right)\right].
\ee
The next step is to take the limit $\d\to\infty$. 
Using  Lemma 4.10 of \cite{ABK_E} we have that $C(\bar\phi,\d)$  
is monotone decreasing in $\d$ and 
$\lim_{\d\to\infty}C(\bar\phi,\d)=C(\bar\phi)$, exists and is strictly positive, where
 \be C(\bar\phi)=\lim_{t\uparrow\infty} \sqrt {\sfrac 2\pi}\int_0^\infty
u(t,y+\sqrt 2 t)ye^{\sqrt 2y}dy.
\ee
 Here $u(t,x)$
 is a solution to the F-KPP equation \eqv(fkpp)
 with initial condition $u(0,x)=1-e^{-\bar\phi (-x)}$.
Using the same monotonicity arguments shows that also  
\be
\lim_{\d\to\infty}\tilde C(Z,C(\bar\phi,\d))=\tilde C(Z,C(\bar\phi)).
\ee
Therefore, taking the limit first as $D\uparrow\infty$ and then $\d\uparrow \infty$ in the left-hand side of  \eqv(seltsam.5), 
we get that 
\bea\Eq(K.7)
\lim_{t\to\infty}\Psi_t(\phi(\cdot+x))
&=&\lim_{\d\uparrow\infty}\lim_{t\to\infty}\Psi^{<\d}_t(\phi(\cdot+x))
\\\nonumber
&=&\lim_{\d\uparrow\infty}\lim_{D\uparrow\infty}
\lim_{t\to\infty} v_{\d,D}(t,m(t)+x)
=\E\left[e^{-\tilde C(Z,C(\bar\phi))\tilde Ze^{-\sqrt 2 x}}\right].
\eea
To see that the constants $\tilde C(Z,C(\bar\phi))$ are strictly positive, one uses the Laplace functionals 
$\Psi_t(\phi)$  are bounded from above by
\be\Eq(fz.30)
\E\left[\exp\left({-\phi\left(\max_{i\leq n(bt)} x_i(bt)+\max_{j\leq n_1((1-b)t)}
x_j^1((1-b)t)-m(t)\right)}\right)\right]
%_{\stackrel{i\leq n(bt)}{j\leq n_i((1-b)t)}} x_i(bt)+x_j^i((1-b)t)-m(t)\right)}\right)\right].
\ee
Here we used that the offspring of any of the particles at time $bt$ has the same law. So the sum of the two maxima in the expression above has the same distribution as the largest descendent at time $t$ off the largest particle at time $bt$.
The limit of Eq. \eqv(fz.30) as $t\uparrow\infty$ exists and is strictly smaller than $1$ by the convergence in law of the recentered maximum of a standard BBM. 
But this implies the positivity of the constants $\tilde C$. Hence a limiting point process exists.
Finally,  one may easily check that the right hand side of \eqv(K.7) coincides with the Laplace functional of the point process defined in \eqv(speed.5) by basically repeating the computations above.
\end{proof}

\begin{remark}
Note that in particular, the structure of the variance profile is contained in the constant $\tilde C(D,Z,C(\bar\phi,\d))$ and that also the information on the structure of the limiting point process is contained in this constant. 
In fact, we see that in all cases we have considered in this paper, the Laplace functional of the limiting process has the form 
\be\Eq(universal.1)
\lim_{t\uparrow\infty} \Psi_t(\phi(\cdot+x)) = \E \exp\left( -C(\phi) M  e^{-\sqrt 2  x}\right),
\ee
where $M$ is a martingale limit (either $Y$ of $Z$) and 
$C$ is a map from the space of positive continuous functions with compact support to the real numbers. This function contains all  the information on 
the specific limiting process. This is compatible with the finding in \cite{MZ} in the case where the speed is a concave function of 
$s/t$. 
The universal form \eqv(universal.1) is thus misleading and without knowledge of the specific form of $C(\phi)$, \eqv(universal.1) contains almost no information. 
\end{remark}

\end{document}